\documentclass[12pt]{amsart}

\usepackage{enumerate, amsmath, amsthm, amsfonts, amssymb, xy,  mathrsfs, graphicx, paralist, fancyvrb, eucal}
\usepackage[usenames, dvipsnames]{xcolor}
\usepackage[margin=1in]{geometry} 
\usepackage[bookmarks, colorlinks=true, linkcolor=blue, citecolor=blue, urlcolor=blue]{hyperref}

\input xy
\xyoption{all}

\numberwithin{equation}{section}
\newtheorem{theorem}[equation]{Theorem}

\newtheorem{proposition}[equation]{Proposition}
\newtheorem{lemma}[equation]{Lemma}
\newtheorem{corollary}[equation]{Corollary}

\newtheorem{stepx}{Step}
\newenvironment{step}[1]
 {\renewcommand\thestepx{#1}\stepx}
 {\endstepx}

\newtheorem{substepx}{Substep}
\newenvironment{substep}[1]
 {\renewcommand\thesubstepx{#1}\substepx}
 {\endsubstepx}

\theoremstyle{definition}
\newtheorem{rmk}[equation]{Remark}
\newenvironment{remark}[1][]{\begin{rmk}[#1] \pushQED{\qed}}{\popQED \end{rmk}}
\newtheorem{eg}[equation]{Example}
\newenvironment{example}[1][]{\begin{eg}[#1] \pushQED{\qed}}{\popQED \end{eg}}
\newtheorem{defn}[equation]{Definition}
\newenvironment{definition}[1][]{\begin{defn}[#1]\pushQED{\qed}}{\popQED \end{defn}}

\makeatletter
\@addtoreset{equation}{section}
\renewcommand{\thesubsection}{%
  \ifnum\c@subsection<1 \@arabic\c@section
  \else \thesection.\@arabic\c@subsection
  \fi
}
\makeatother

\newcommand{\bB}{\mathbf{B}}

\newcommand{\cC}{\mathcal{C}}

\newcommand{\cD}{\mathcal{D}}

\newcommand{\bE}{\mathbf{E}}

\newcommand{\rE}{\mathrm{E}}

\newcommand{\bF}{\mathbf{F}}

\newcommand{\rF}{\mathrm{F}}

\newcommand{\bG}{\mathbf{G}}

\newcommand{\bH}{\mathbf{H}}

\newcommand{\rH}{\mathrm{H}}

\newcommand{\bK}{\mathbf{K}}

\newcommand{\fM}{\mathfrak{M}}

\newcommand{\bN}{\mathbf{N}}

\newcommand{\fN}{\mathfrak{N}}

\newcommand{\cO}{\mathcal{O}}

\newcommand{\bP}{\mathbf{P}}

\newcommand{\fP}{\mathfrak{P}}
\newcommand{\rP}{\mathrm{P}}

\newcommand{\bQ}{\mathbf{Q}}

\newcommand{\bS}{\mathbf{S}}

\newcommand{\bT}{\mathbf{T}}

\newcommand{\bU}{\mathbf{U}}

\newcommand{\cW}{\mathcal{W}}

\newcommand{\bZ}{\mathbf{Z}}

\newcommand{\fa}{\mathfrak{a}}

\newcommand{\bk}{\mathbf{k}}

\newcommand{\fu}{\mathfrak{u}}

\renewcommand{\phi}{\varphi}
\renewcommand{\emptyset}{\varnothing}

\newcommand{\ol}[1]{\overline{#1}}
\newcommand{\ul}[1]{\underline{#1}}

\newcommand{\arxiv}[1]{\href{http://arxiv.org/abs/#1}{{\tt arXiv:#1}}}

\makeatletter
\def\Ddots{\mathinner{\mkern1mu\raise\p@
\vbox{\kern7\p@\hbox{.}}\mkern2mu
\raise4\p@\hbox{.}\mkern2mu\raise7\p@\hbox{.}\mkern1mu}}
\makeatother

\renewcommand{\hom}{\operatorname{Hom}}

\DeclareMathOperator{\Aut}{Aut}

\DeclareMathOperator{\Mod}{Mod}

\newcommand{\GL}{\mathbf{GL}}
\newcommand{\SL}{\mathbf{SL}}

\newcommand{\Set}[2]{\ensuremath{\{\text{#1 $|$ #2}\}}}

\newcommand{\FI}{\mathbf{FI}}
\newcommand{\OI}{\mathbf{OI}}
\newcommand{\OS}{\mathbf{OS}}

\newcommand{\OVI}{\mathbf{OVI}}

\newcommand{\pt}{\mathrm{pt}}
\newcommand{\id}{\mathrm{id}}
\newcommand{\op}{\mathrm{op}}
\DeclareMathOperator{\Inc}{Inc}
\DeclareMathOperator{\Rep}{Rep}
\DeclareMathOperator{\ini}{in}
\DeclareMathOperator{\fin}{in}
\DeclareMathOperator{\colim}{colim}

\newtheorem*{theorem*}{Theorem}

\DeclareMathOperator{\Hom}{Hom}
\DeclareMathOperator{\Ind}{Ind}

\newcommand{\ox}{\overline{x}}
\newcommand{\oy}{\overline{y}}
\newcommand{\orr}{\overline{r}}

\DeclareMathOperator{\tors}{tor}

\title{Stability in the homology of unipotent groups}
\author{Andrew Putman}
\address{Department of Mathematics, University of Notre Dame, South Bend, IN}
\email{\href{mailto:andyp@nd.edu}{andyp@nd.edu}}
\urladdr{\url{http://www.nd.edu/~andyp/}}
\thanks{AP was partially supported by NSF DMS-1737434.}

\author{Steven V Sam}
\address{Department of Mathematics, University of Wisconsin, Madison, WI}
\email{\href{mailto:svs@math.wisc.edu}{svs@math.wisc.edu}}
\urladdr{\url{http://math.wisc.edu/~svs/}}
\thanks{SS was partially supported by NSF DMS-1500069, DMS-1651327, and a Sloan research fellowship.}

\author{Andrew Snowden}
\address{Department of Mathematics, University of Michigan, Ann Arbor, MI}
\email{\href{mailto:asnowden@umich.edu}{asnowden@umich.edu}}
\urladdr{\url{http://www-personal.umich.edu/~asnowden/}}
\thanks{AS was partially supported by NSF DMS-1303082, DMS-1453893, and a Sloan research fellowship.}


\date{December 19, 2018}

\begin{document}

\begin{abstract}
Let $R$ be a (not necessarily commutative) ring whose additive group is finitely generated and let $U_n(R) \subset \GL_n(R)$ be the group of upper-triangular unipotent matrices over $R$. We study how the homology groups of $U_n(R)$ vary with $n$ from the point of view of representation stability.  Our main theorem asserts that if for each $n$ we have representations $M_n$ of $U_n(R)$ over a ring $\bk$ that are appropriately compatible and satisfy suitable finiteness hypotheses, then the rule $[n] \mapsto \rH_i(U_n(R),M_n)$ defines a finitely generated $\OI$-module.  As a consequence, if $\bk$ is a field then $\dim \rH_i(U_n(R),\bk)$ is eventually equal to a polynomial in $n$.  We also prove similar results for the Iwahori subgroups of $\GL_n(\cO)$ for number rings $\cO$. 
\end{abstract}

\newpage

\maketitle
\tableofcontents

\section{Introduction}

\subsection{Homology of unipotent groups}

Groups of the form $G(R)$, with $G$ a linear algebraic group over a ring $R$, are among the most common and important groups encountered in mathematics. It is therefore a natural problem to understand their group homology, as homology is one of the most important invariants of a group. In the case where $G$ is reductive, this problem has been studied intensively and much is known.  See, for instance, \cite{BorelStable} for $G$ a classical group and
$R$ a number ring, and \cite{QuillenFinite} for $G = \GL_n$ and $R$ a finite field.  These
computations are closely connected to algebraic K-theory.

On the other hand, when $G$ is a unipotent group, comparatively little is known. In fact, the class of unipotent groups is fairly wild, so there might not be too much one can say in complete generality. Let $U_n \subset \GL_n$ be the group of upper-unitriangular matrices. These are perhaps the most important unipotent groups; for example, Engel's Theorem 
\cite[Corollary I.4.8]{BorelAlgebraic} shows 
that any unipotent group embeds into one of them.
Nonetheless, the homology of even these groups is poorly understood. The purpose of this paper is to establish some new results in this direction.

To illustrate the difficulties in computing the homology of $U_n(R)$, let us consider the first few cases. We take $R=\bF_p$ for simplicity. The group $U_1(\bF_p)$ is trivial. The group $U_2(\bF_p)$ is simply isomorphic to the additive group of $\bF_p$, i.e., $\bZ/p\bZ$, and the homology of this group is known (it is $\bZ/p\bZ$ in even degrees and vanishes in odd degrees). The group $U_3(\bF_p)$ is a non-abelian group of order $p^3$. It fits into an exact sequence
\begin{displaymath}
1 \to \bZ/p\bZ \to U_3(\bF_p) \to (\bZ/p\bZ)^2 \to 1,
\end{displaymath}
where the left $\bZ/p\bZ$ is the center of $U_3(\bF_p)$. We therefore have a spectral sequence (the Leray--Serre spectral sequence) that computes the homology of $U_3(\bF_p)$ in terms of the homology of the outer groups:
\begin{displaymath}
\rE^2_{p,q} = \rH_p((\bZ/p\bZ)^2, \rH_q(\bZ/p\bZ, \bZ)) \implies \rH_{p+q}(U_3(\bF_p), \bZ).
\end{displaymath}
The action of $(\bZ/p\bZ)^2$ on $\rH_q(\bZ/p\bZ, \bZ)$ is trivial, and so the groups on the $\rE^2$ page are easy to compute. However, it is less clear what the differentials are on the $\rE^2$ page, much less on subsequent pages, and so it is not obvious how to actually compute the homology of $U_3(\bF_p)$ explicitly from this spectral sequence.

The analysis of $U_3(\bF_p)$ we have just made, discouraging though it may be, does highlight a general theoretical approach to studying the homology of $U_n(\bF_p)$: this group is nilpotent, so one can break it up into abelian groups and then use the resulting spectral sequences to study its homology. Of course, this approach becomes increasingly complicated as $n$ grows, and there is probably little chance of understanding the spectral sequences in an explicit way in general.

The main point of this paper is that, although these spectral sequences become increasingly complicated, they exhibit a kind of regularity as $n$ varies. The precise formulation of this statement uses the language of representation stability, and requires some preliminaries, so for the moment we simply give a sample application to the main objects of interest:

\begin{theorem} \label{mainthm1}
Let $R$ be a (not necessarily commutative) ring whose additive group is finitely generated and let $\bk$ be a field. For all $i \ge 0$, there exists some $f_i(t) \in \bQ[t]$ such that $\dim \rH_i(U_n(R), \bk) = f_i(n)$ for $n \gg 0$.
\end{theorem}

For the ring $R$ in the theorem, one could take a finite field, or a number ring, or the ring of $2 \times 2$ matrices over one of these rings, for example.


\begin{example}
The case $R=\bZ$ and $\bk = \bQ$ of Theorem \ref{mainthm1} follows from work of Dwyer \cite[Theorem 1.1]{dwyer}.  He shows that the dimension of $\rH_i(U_n(\bZ),\bQ)$ is the number of permutations in $S_n$ with length $i$, where the length of a permutation $\sigma$ is the number of pairs $i<j$ such that $\sigma(i) > \sigma(j)$. Denote this number by $I(i,n)$. We claim that $n \mapsto I(i,n)$ is a polynomial of degree $i$ for $n \gg 0$. As an aside, this shows that the degree of the polynomials $f_i(t)$ in Theorem \ref{mainthm1} cannot be bounded as we let $i$ vary. We prove the claim by induction. For $i=1$, we have $I(1,n) = n-1$ for $n > 0$. In general, we have the identity 
\[
\sum_{i \ge 0} I(i,n) q^i = (1+q)(1+q+q^2) \cdots (1+q+q^2 + \cdots + q^{n-1});
\]
see \cite[Corollary 1.3.13]{stanleyEC1}.  It follows that $I(i,n) - I(i,n-1) = \sum_{j=0}^{i-1} I(j, n-1)$ for $n>i$. By induction, the right hand side is a polynomial of degree $i-1$ for $n \gg 0$. Hence $I(i,n)$ is a polynomial of degree $i$ for $n \gg 0$, as claimed.
\end{example}

\subsection{Main results}

Our main result is a refined version of Theorem~\ref{mainthm1} where we allow systems of non-trivial coefficients and give a stronger conclusion. This additional generality is interesting in its own right, but is required even if one is ultimately only interested in the case of trivial coefficients. Indeed, our general approach essentially relates the $i$th homology group of some system of coefficients to lower homology groups of some auxiliary systems, and the auxiliary systems can be non-trivial even if the initial system is trivial.

To formulate this general theorem, we must make sense of a ``system'' of representations of $U_n(R)$. For this, we introduce the category $\OVI(R)$. An object of $\OVI(R)$ is a finite rank free $R$-module equipped with a totally ordered basis. A morphism of $\OVI(R)$ is a map of $R$-modules that is upper-triangular with respect to the distinguished ordered bases (see \S \ref{ss:defovi}). An {\bf $\OVI(R)$-module} over a commutative ring $\bk$ is a functor $\OVI(R) \to \Mod_{\bk}$. Every object in $\OVI(R)$ is isomorphic to $R^n$ equipped with its standard basis for some $n$, and the automorphism group of this object is the group $U_n(R)$. Thus an $\OVI(R)$-module $M$ gives rise to a sequence $\{M_n\}_{n \ge 0}$, where $M_n=M(R^n)$ is a representation of $U_n(R)$, and therefore provides a reasonable notion of a system of $U_n(R)$ representations. We are primarily interested in \emph{finitely generated} $\OVI(R)$-modules (see \S \ref{section:catgen} for the definition): indeed, it is only reasonable to expect uniform behavior of the homology in this case.

\begin{example}
(a) We have a constant $\OVI(R)$-module given by $R^n \mapsto \bk$ for all $n$. Thus the sequence of trivial representations of $U_n(R)$ forms a ``system'' in our sense. (b) Suppose $R=\bk$. We then have an $\OVI(R)$-module given by $R^n \mapsto R^n$. We thus see that, in this case, the sequence of standard representations of $U_n(R)$ forms a ``system.'' Both examples are finitely generated.
\end{example}

Let $M$ be an $\OVI(R)$-module, and fix $i \ge 0$. For each $n$, we consider the homology group $\rH_i(U_n(R), M_n)$. The various $M_n$'s are related by the $\OVI(R)$-module structure, and this should lead to relationships between these homology groups. We now examine this. Letting $[n]$ denote the ordered set $\{1,\ldots,n\}$, if $[n] \to [m]$ is an order-preserving injection of finite sets then there is an associated morphism $R^n \to R^m$ in $\OVI(R)$. This gives a map $M_n \to M_m$, which induces a map $\rH_i(U_n(R), M_n) \to \rH_i(U_m(R), M_m)$. This suggests that $[n] \mapsto \rH_i(U_n(R), M_n)$ defines an $\OI$-module, where $\OI$ is the category whose objects are finite totally ordered sets and whose morphisms are order-preserving injections. We show that this is indeed the case, and denote this $\OI$-module by $\bH_i(\bU, M)$. (We note that $\OI$-modules are close relatives of the well-known $\FI$-modules introduced by Church, Ellenberg, and Farb \cite{fimodule}.)

We can now state our main theorem:

\begin{theorem} \label{mainthm2}
Let $R$ be a ring whose additive group is finitely generated, let $\bk$ be a noetherian commutative ring, and $M$ be a finitely generated $\OVI(R)$-module over $\bk$. Then $\bH_i(\bU, M)$ is a finitely generated $\OI$-module over $\bk$ for all $i \geq 0$.
\end{theorem}

\noindent
Theorem~\ref{mainthm1} follows immediately from this theorem by taking $M_n$ to be the trivial representation of $U_n(R)$ for all $n$ and appealing to the fact that a finitely generated $\OI$-module over a field has eventually polynomial dimension (see Proposition~\ref{prop:oidim} below).

\subsection{The noetherian result}

As stated, to prove Theorem~\ref{mainthm2} we relate the homology of the $\OVI(R)$-module $M$ to the homology of certain auxiliary coefficient systems constructed by various means. To ensure that these auxiliary systems are finitely generated, we require the following noetherian result, which is the primary technical result of this paper:

\begin{theorem} \label{mainthm3}
Let $R$ be a ring whose additive group is finitely generated and let $\bk$ be a noetherian commutative ring. Then the category of $\OVI(R)$-modules over $\bk$ is locally noetherian, that is, any submodule of a finitely generated module is finitely generated.
\end{theorem}

\noindent
Theorem~\ref{mainthm3} differs from much previous work on categories of $R$-modules in the setting of representation stability (such as \cite{putman-sam} and \cite{catgb}) in that it allows the ring $R$ to be infinite. In the previous work, the automorphism groups in the categories under consideration were $\GL_n(R)$, and finiteness of $R$ is necessary since the group algebra of $\GL_n(R)$ is not noetherian if $R$ is infinite. In our situation, the automorphism groups are $U_n(R)$. When the additive group of $R$ is finitely generated, these groups are virtually polycyclic, and a classical result of Philip Hall \cite{hall} says that group rings of virtually polycyclic groups are noetherian.  Our proof of Theorem \ref{mainthm3} is inspired in part by Hall's proof of this fact.

\begin{remark}
It is easy to see that Theorem~\ref{mainthm3} is false if the additive group of $R$ is not finitely generated (see \S \ref{section:converse}).
\end{remark}

\begin{remark}
When the ring $R$ is finite, we in fact show that the category of $\OVI(R)$-modules is quasi-Gr\"obner in the sense of \cite[\S 4]{catgb}, which implies local noetherianity (but is stronger). In the general case, we do not show that the category of $\OVI(R)$-modules is quasi-Gr\"obner (and expect that it is not), and the proof of local noetherianity is far more difficult.
\end{remark}

%

\subsection{Application to Iwahori groups}

Let $\cO$ be a number ring and let $\bk$ be a commutative noetherian ring.  A classical result of van der Kallen \cite{vanderkallen} says that the homology of the group $\GL_n(\cO)$ stabilizes: for any fixed $i$ the canonical map
\[
\rH_i(\GL_n(\cO),\bk) \rightarrow \rH_i(\GL_{n+1}(\cO),\bk)
\]
is an isomorphism for $n \gg 0$.  In particular, if $\bk$ is a field then the dimension of $\rH_i(\GL_n(\cO),\bk)$ is eventually constant.

Now let $\fa$ be a nonzero proper ideal in $\cO$ and let $\GL_n(\cO,\fa)$ be the principal congruence subgroup of level $\fa$, i.e.,\ the subgroup of $\GL_n(\cO)$ consisting of matrices that are congruent to the identity modulo $\fa$.  The homology of these groups does not stabilize; for instance, for $\ell \geq 2$ and $n \geq 3$ the abelianization of $\GL_n(\bZ,\ell\bZ)$ is $(\bZ/\ell)^{n^2-1}$ (see \cite{LeeSzczarba}). Building on work of the first author \cite{putman}, Church--Ellenberg--Farb--Nagpal \cite{CEFN} proved instead that the homology of $\GL_n(\cO,\fa)$ satisfies a version of representation stability: the rule $[n] \mapsto \rH_i(\GL_n(\cO,\fa), \bk)$ defines a finitely generated $\FI$-module. Consequently, when $\bk$ is a field, the dimension is eventually polynomial.

The Iwahori subgroup $\GL_{n,0}(\cO,\fa)$ is the subgroup of $\GL_n(\cO)$ consisting of matrices that are upper-triangular modulo $\fa$. Using Theorem~\ref{mainthm2}, we prove an analog of Church--Ellenberg--Farb--Nagpal's result for $\GL_{n,0}(\cO,\fa)$:

\begin{theorem}
\label{maintheorem:iwahori}
Let $\cO$ be a number ring, let $\fa \subset \cO$ be a nonzero proper ideal, and let $\bk$ be a commutative noetherian ring.  Then the following hold for all $i \geq 0$.
\begin{compactitem}
\item The rule $[n] \mapsto \rH_i(\GL_{n,0}(\cO,\fa), \bk)$ defines a finitely generated $\OI$-module over $\bk$. 
\item If $\bk$ is a field then there is a polynomial $f \in \bQ[t]$ such that $\dim \rH_i(\GL_{n,0}(\cO,\fa), \bk)=f(n)$ for $n \gg 0$.
\end{compactitem}
\end{theorem}

\subsection{Outline}

In \S \ref{s:cat} we review generalities on modules over categories. In \S \ref{s:oi} we introduce the category $\OI$ and its variants $\OI(d)$ and establish basic results about them. In \S \ref{s:ovi} we introduce the category $\OVI(R)$ and its variants $\OVI(R,d)$ and establish basic results about them. In \S \ref{s:noeth}, we prove the main noetherianity result for $\OVI(R)$ (Theorem~\ref{mainthm3}). In \S \ref{s:hovi} we prove the main result of the paper (Theorem~\ref{mainthm2}). Finally, in \S \ref{s:iwahori} we prove Theorem \ref{maintheorem:iwahori}.

\subsection{Notation}

Throughout, $\bk$ denotes a commutative ring, typically noetherian.  
Unless otherwise specified, $1 \neq 0$ in all of our rings.
For a fixed category $\cC$, we write $\ul{\bk}$ for the constant functor $\cC \to \Mod_{\bk}$ taking everything to $\bk$ and all morphisms to the identity. We let $B_n \subset \GL_n$ be the group of upper-triangular matrices, and $U_n \subset B_n$ the subgroup where the diagonal entries are equal to~1. We use $R$ to denote the ring appearing in the definition of $\OVI(R)$, and that is typically plugged in to $U_n$ or $B_n$. We generally do not require it to be commutative.
We set $[0] = \emptyset$, and if $n$ is a positive integer, then $[n]$ denotes the set $\{1,\dots,n\}$.

\subsection*{Acknowledgments}

We thank Benjamin Steinberg for pointing out a significant simplification to the proof of 
Lemma \ref{lemma:surjectfree}.

\section{Representations of categories}
\label{s:cat}

\subsection{Generalities}
\label{section:catgen}

Let $\cC$ be a category and let $\bk$ be a noetherian commutative ring. A {\bf $\cC$-module} over $\bk$ is a functor $M\colon \cC \to \Mod_{\bk}$.  For an object $x \in \cC$, we denote by $M_x$ the image of $x$ under $M$.  Denote the category of $\cC$-modules by $\Rep_\bk(\cC)$. It is an abelian category. For each $x \in \cC$, we define a $\cC$-module $P_x$ via the formula $(P_x)_y = \bk[\hom(x,y)]$. One easily sees that for any $\cC$-module $M$ one has a natural identification $\Hom(P_x, M)=M_x$. It follows that $P_x$ is a projective $\cC$-module; we call it the {\bf principal projective} at $x$. A general $\cC$-module $M$ is finitely generated if and only if there exists a surjection $\bigoplus_{i=1}^k P_{x_i} \to M$ for some $x_1,\ldots,x_k \in \cC$. A $\cC$-module is said to be {\bf noetherian} if all of its submodules are finitely generated, and the category $\Rep_{\bk}(\cC)$ is said to be {\bf locally noetherian} if all finitely generated objects are noetherian.

If $\Phi \colon \cC \to \cD$ is a functor and $M$ is a $\cD$-module then the {\bf pullback} of $M$ along $\Phi$, denoted $\Phi^*(M)$, is the $\cC$-module defined via the formula $\Phi^*(M) = M \circ \Phi$, so that $\Phi^*(M)_x = M_{\Phi(x)}$.  We now review how the pullback operation interacts with finite generation. The following definition is \cite[Def.~3.2.1]{catgb}:

\begin{definition} \label{def:propF}
We say that a functor $\Phi \colon \cC \to \cD$ satisfies {\bf property (F)} if the following condition holds for all $y \in \cD$.  There exist finitely many objects $x_1, \ldots, x_n \in \cC$ together with morphisms $f_i \colon y \to \Phi(x_i)$ in $\cD$ with the following property: for any $x \in \cC$ and any morphism $f \colon y \to \Phi(x)$ in $\cD$, there exists an $i$, and a morphism $g \colon x_i \to x$ in $\cC$, such that $f = \Phi(g) \circ f_i$.
\end{definition}

\begin{definition}
A category $\cC$ satisfies property~(F) if the diagonal $\cC \to \cC \times \cC$ satisfies property~(F).
\end{definition}

The importance of these definitions is due to the following results.

\begin{proposition} \label{prop:fpullback}
A functor $\Phi \colon \cC \to \cD$ satisfies property~{\rm (F)} if and only if $\Phi^*(M)$ is a finitely generated $\cC$-module for all finitely generated $\cD$-modules $M$.
\end{proposition}

\begin{proof}
See \cite[Prop.~3.2.3]{catgb}.
\end{proof}

Recall that a functor $\Phi\colon \cC \to \cD$ is {\bf essentially surjective} if for all $y \in \cD$, there exists some $x \in \cC$ such that $\Phi(x)$ is isomorphic to $y$.

\begin{proposition} \label{prop:noeth}
Let $\cC$ be a category such that $\Rep_{\bk}(\cC)$ is locally noetherian and let $\Phi \colon \cC \to \cD$ be an essentially surjective functor satisfying property~{\rm (F)}. Then $\Rep_{\bk}(\cD)$ is locally noetherian.
\end{proposition}

\begin{proof}
See \cite[Cor.~3.2.5]{catgb}.
\end{proof}

If $\cC$ is a category and $M_1$ and $M_2$ are $\cC$-modules, then we define $M_1 \otimes M_2$ to be the $\cC$-module defined by the formula $(M_1 \otimes M_2)_x = (M_1)_x \otimes (M_2)_x$ for all $x \in \cC$.

\begin{proposition} \label{prop:pointwisetensor}
Let $\cC$ be a category that satisfies property~{\rm (F)} and let $M$ and $N$ be finitely generated $\cC$-modules.  Then $M \otimes N$ is finitely generated.
\end{proposition}

\begin{proof}
See \cite[Prop.~3.3.2]{catgb}.
\end{proof}

We require a slight variant of the above proposition. We say that a $\cC$-module $M$ is {\bf generated in finite degrees} if there exist $x_1, \ldots, x_k \in \cC$ such that $M$ is generated by the $M_{x_i}$, that is, the canonical map $\bigoplus_{i=1}^k M_{x_i} \otimes P_{x_i} \to M$ is surjective. Note that if $M$ is generated in finite degrees and $M_x$ is a finitely generated $\bk$-module for all $x \in \cC$ then $M$ is finitely generated.

\begin{proposition} \label{prop:finddeg}
Let $\cC$ be a category that satisfies property~{\rm (F)} and let $M$ and $N$ be $\cC$-modules generated in finite degrees. Then $M \otimes N$ is generated in finite degrees.
\end{proposition}

\begin{proof}
Observe that (a) a finite sum of $\cC$-modules generated in finite degrees is generated in finite degrees; (b) if $K$ is a $\cC$-module generated in finite degrees and $U$ is any $\bk$-module then $U \otimes K$ is generated in finite degrees; (c) any quotient of a $\cC$-module generated in finite degrees is generated in finite degrees. Now, choose surjections $\bigoplus_{i=1}^k V_i \otimes P_{x_i} \to M$ and $\bigoplus_{j=1}^{\ell} W_j \otimes P_{y_j} \to N$, where the $x_i$ and $y_j$ are objects of $\cC$ and the $V_i$ and $W_j$ are $\bk$-modules (one can take $V_i=M_{x_i}$ and $W_j=N_{y_j}$). We thus have a surjection
\begin{displaymath}
\bigoplus_{i,j} V_i \otimes W_j \otimes P_{x_i} \otimes P_{y_j} \to M \otimes N.
\end{displaymath}
Since $\cC$ satisfies property~(F), each $P_{x_i} \otimes P_{y_j}$ is finitely generated (Proposition~\ref{prop:pointwisetensor}). Thus each term in the sum is generated in finite degrees by (b); since the sum is finite, it is generated in finite degree by (a); and so we conclude $M \otimes N$ is generated in finite degrees by (c).
\end{proof}

Now we recall the notion of a Gr\"obner category. See \cite[\S 4.3]{catgb} for more details. 

\begin{definition}
Let $\cC$ be an essentially small category, i.e.,\ there exists a set $I$ containing a unique representative of each isomorphism class in $\cC$.
For $x \in \cC$, define $|S_x| = \amalg_{y \in I} \hom(x,y)$.  Partially order $|S_x|$ by defining $f \preceq g$ if there exists a morphism $h$ such that $g = h f$. We say that $\cC$ is {\bf Gr\"obner} if the following holds for all $x \in \cC$.
\begin{itemize}
\item The poset $(|S_x|,\preceq)$ is noetherian.
\item $|S_x|$ admits a total ordering $\leq$ with the following two properties.
\begin{itemize}
\item The ordering $\leq$ is compatible with left composition, i.e.,\ $f \leq g$ implies $hf \leq hg$.
\item The restriction of $\leq$ to each $\hom(x,y)$ is a well-ordering.
\end{itemize}
\end{itemize}
We say that $\cC$ is {\bf quasi-Gr\"obner} if there exists a Gr\"obner category $\cC'$ and an essentially surjective functor $\cC' \to \cC$ satisfying property~(F).
\end{definition}

The key result about quasi-Gr\"ober categories is the following \cite[Theorem 4.3.2]{catgb}:

\begin{theorem}
\label{theorem:quasigrobner}
Let $\cC$ be a quasi-Gr\"obner category. Then for any noetherian commutative ring $\bk$, the category $\Rep_{\bk}(\cC)$ is locally noetherian.
\end{theorem}

\subsection{Kan extension} \label{ss:kan}

Let $\Phi \colon \cC \to \cD$ be a functor. The pullback functor $\Phi^*$ on modules admits a left adjoint $\Phi_!$ called the {\bf left Kan extension}.  It also admits a right adjoint $\Phi_*$ called the {\bf right Kan extension}, but we will not need this.  

The left Kan extension can be described explicitly as follows. Let $y$ be an object of $\cD$. Define a category $\cC_{/y}$ as follows. An object of $\cC_{/y}$ is a pair $(x, f)$, where $x$ is an object of $\cC$ and $f \colon \Phi(x) \to y$ is a morphism in $\cD$. A morphism $(x',f') \to (x,f)$ in $\cC_{/y}$ is a morphism $g \colon x' \to x$ in $\cC$ such that $f'=f \circ \Phi(g)$. Suppose now that $M$ is a $\cC$-module over $\bk$.  For $y \in \cD$, define $M|_{\cC_{/y}}$ to be the $\cC_{/y}$-module defined via the formula $(M|_{\cC_{/y}})_{(x,f)} = M_x$.  We then have
\begin{displaymath}
\Phi_!(M)_y = \colim(M \vert_{\cC_{/y}}).
\end{displaymath}
That is, the value of $\Phi_!(M)$ on $y$ is the colimit of the functor $M|_{\cC_{/y}} \colon \cC_{/y} \to \Mod_{\bk}$. In certain cases, there is an even nicer description:

\begin{proposition} \label{prop:kanind}
Let $\Phi\colon \cC \to \cD$ be a faithful functor.  Assume that for all $x',x \in \cC$, the $\Aut(\Phi(x))$-orbit of any element of $\Hom_{\cD}(\Phi(x'), \Phi(x))$ contains an element of the form $\Phi(f)$ for some $f \in \Hom_{\cC}(x',x)$ that is unique up to the action of $\Aut(x)$. Let $M$ be a $\cC$-module. Then for all $x \in \cC$ we have a canonical isomorphism
\begin{displaymath}
\Psi_!(M)_{\Phi(x)}=\Ind_{\Aut(x)}^{\Aut(\Phi(x))}(M_{x}).
\end{displaymath}
\end{proposition}

\begin{proof}
Let $\{h_i\}_{i \in I}$ be a set of coset representatives for $\Aut(\Phi(x))/\Aut(x)$.  For each $i \in I$, we thus have an object $(x, h_i)$ of $\cC_{/\Phi(x)}$.  Consider an object $(x',g)$ of $\cC_{/\Phi(x)}$.  To prove the proposition, it is enough to prove that there is a unique $i \in I$ and a unique morphism $(x',g) \rightarrow (x,h_i)$ of $\cC_{/\Phi(x)}$.

By definition, $g$ is a morphism $\Phi(x') \to \Phi(x)$ in $\cD$.  By assumption, we can factor $g$ as $h \Phi(f)$ for some $h \in \Aut(\Phi(x'))$ and some $f \in \Hom_{\cC}(x',x)$. Moreover, this factorization is unique up to the action of $\Aut(x)$. It follows that there is a unique factorization of the form $h_i \Phi(f)$. The morphism $f$ now furnishes a map $(x',g) \to (x,h_i)$ in $\cC_{/\Phi(x)}$.  It is clear from the discussion that this is the unique $i$ for which there is such a morphism, and that $f$ is the unique such morphism.
\end{proof}

Left Kan extensions can be used to construct principal projectives, as follows. Let $x \in \cC$, let $\pt$ be the point category (one object, one morphism), and let $i_x \colon \pt \to \cC$ be the functor taking the object of $\pt$ to $x$. Regarding $\bk$ as a $\pt$-module, we have $(i_x)_!(\bk)=P_x$. Indeed, if $M$ is a $\cC$-module, then by definition
\begin{displaymath}
\Hom_{\Rep_\bk(\cC)}((i_x)_!(\bk), M)=\Hom_{\Rep_\bk(\pt)}(\bk, i_x^*(M))=M_x,
\end{displaymath}
and thus $(i_x)_!(\bk)$ represents the same functor as $P_x$.

Return now to the setting of a functor $\Phi \colon \cC \to \cD$. Put $y=\Phi(x)$. Then $\Phi \circ i_x=i_y$, so
\begin{align} \label{eqn:kan-proj}
P_y=(i_y)_!(\bk)=\Phi_!((i_x)_!(\bk))=\Phi_!(P_x).
\end{align}
We thus see that the left Kan extension takes principal projectives to principal projectives. Since $\Phi_!$ is right exact, it follows from this that $\Phi_!$ takes finitely generated $\cC$-modules to finitely generated $\cD$-modules.

\subsection{$\cC$-groups and their representations}

Let $\cC$ be a category. A {\bf $\cC$-group} is a functor from $\cC$ to the category of groups. Fix a $\cC$-group $\bG$. A {\bf $\bG$-module} over $\bk$ is a $\cC$-module $M$ equipped with a $\bk$-linear action of $\bG_x$ on $M_x$ for all $x \in \cC$, such that for all morphisms $f \colon x \to y$ in $\cC$ the induced morphism $f_* \colon M_x \to M_y$ is compatible with the actions via the induced homomorphism $f_* \colon \bG_x \to \bG_y$.  In other words, for $m \in M_x$ and $g \in \bG_x$ we have $f_*(gm)=f_*(g) f_*(m)$. The category $\Rep_{\bk}(\bG)$ of $\bG$-modules is a Grothendieck abelian category.

Let $M$ be a $\bG$-module. For $x \in \cC$, let $\bH_i(\bG,M)_x$ be the group homology $\rH_i(\bG_x, M_x)$. If $f \colon x \to y$ is a morphism in $\cC$, then the induced morphisms $f_*\colon \bG_x \rightarrow \bG_y$ and $f_*\colon M_x \rightarrow M_y$ together induce a morphism $f_* \colon \bH_i(\bG,M)_x \to \bH_i(\bG,M)_y$.  This yields a $\cC$-module structure on $\bH_i(\bG,M)$.  If $\bk$ is a commutative ring, then we will denote by $\ul{\bk}$ the constant $\cC$-module defined via the formula $\ul{\bk}_x = \bk$.  We then have $\bH_i(\bG,\ul{\bk})_x = \rH_i(\bG_x, \bk)$.

The following proposition concerns the homology of a semi-direct product of $\cC$-groups.

\begin{proposition} \label{prop:hochserre}
Let $\bG$ and $\bE$ be $\cC$-groups, and let $\pi \colon \bG \to \bE$ and $\iota \colon \bE \to \bG$ be morphisms of $\cC$-groups such that $\pi \circ \iota = \id$. Let $\bK=\ker(\pi)$, which is also a $\cC$-group. Then we have the following:
\begin{enumerate}[\rm \indent (1)]
\item $\bH_i(\bK, \ul{\bk})$ is naturally an $\bE$-module.
\item As a $\cC$-module, $\bH_i(\bE, \ul{\bk})$ is a direct summand of $\bH_i(\bG, \ul{\bk})$ via $\iota_*$ and $\pi_*$.
\item Write $\bH_r(\bG, \ul{\bk})=\rH_r(\bE, \ul{\bk}) \oplus M$ as in {\rm (2)}. Then $M$ admits a $\cC$-module filtration where the graded pieces are subquotients of $\bH_i(\bE, \bH_{r-i}(\bK, \ul{\bk}))$ with $0 \le i \le r-1$.
\end{enumerate}
\end{proposition}

\begin{proof}
(1) The conjugation action of $\bG$ on $\bK$ is $\cC$-linear. On homology, $\bK$ acts trivially, and hence this action descends to give an $\bE$-module structure on $\bH_i(\bK, \ul{\bk})$.

(2) This is clear.

(3) For $x \in \cC$ we have a short exact sequence of groups $1 \to \bK_x \to \bG_x \to \bE_x\to 1$, which gives a Hochschild--Serre spectral sequence
\[
\rE^2_{p,q} = \bH_p(\bE_x, \bH_q(\bK_x, \bk)) \Longrightarrow \bH_{p+q}(\bG_x, \bk).
\]
The spectral sequence is functorial in $x$, and so we get a spectral sequence of $\cC$-modules
\[
\rE^2_{p,q} = \bH_p(\bE, \bH_q(\bK, \ul{\bk})) \Longrightarrow \bH_{p+q}(\bG, \ul{\bk}).
\]
In particular, $\bH_r(\bG, \ul{\bk})$ has a filtration by subquotients of the terms $\rE^2_{i,r-i}$. The edge map $\bH_r(\bG, \ul{\bk}) \to \bH_r(\bE, \rH_0(\bK, \ul{\bk}))$ coincides with the map on $\rH_r$ induced by $\pi$ (see \cite[\S 6.8.2]{weibel}) which we know is a split surjection, so the kernel $M$ has a filtration by subquotients of $\rE^2_{i,r-i}$ for $0 \le i \le r-1$.
\end{proof}

\section{The category \texorpdfstring{$\OI$}{OI} and variants}
\label{s:oi}

\subsection{Definitions and first results}

Let $\OI$ be the category whose objects are finite totally ordered sets and whose morphisms are order-preserving injections. For a non-negative integer $d$, we define a variant $\OI(d)$ as follows. An object of $\OI(d)$ is a pair $(S, \lambda)$ where $S$ is a totally ordered set and $\lambda=(\lambda_1<\cdots<\lambda_d)$ is an increasing $d$-tuple in $S$. A morphism $(S,\lambda) \to (T,\mu)$ in $\OI(d)$ is an order-preserving injection $f \colon S \to T$ satisfying $f(\lambda)=\mu$. Note that $\OI = \OI(0)$. There is a functor $\Phi \colon \OI(d) \to \OI$ given by $\Phi(S,\lambda)=S$. We will continue to use the notation $\Phi$ for this functor throughout the paper (and use it for all values of $d$).

\begin{remark}
We introduce $\OI(d)$ to help us study an analogous category $\OVI(R,d)$, the motivation for which is discussed in Remark \ref{remark:ovidmotivation} below.
\end{remark}

Recall that $[n]$ denotes the ordered set $\{1,\ldots,n\}$.  Given an $\OI$-module $M$,
we will write $M_n$ for $M_{[n]}$.  The category $\OI$ is equivalent to its full
subcategory spanned by the $[n]$, so the data of an $\OI$-module $M$ is equivalent
to the data of the $M_n$ together with the maps $f_{\ast}\colon M_n \rightarrow M_{m}$
induced by the order preserving maps $f\colon [n] \rightarrow [m]$.  Similarly, if $M$
is an $\OI(d)$-module and $\lambda$ is an increasing $d$-tuple in $[n]$, then we will
write $M_{n,\lambda}$ for $M_{([n],\lambda)}$.  
  
\begin{proposition} \label{proposition:oidoi}
There is an equivalence of categories $\OI(d)=\OI^{d+1}$.
\end{proposition}

\begin{proof}
Let $(S, \lambda)$ be an object of $\OI(d)$. For $1 \le i \le d+1$, let $S_i$ be the set of elements $x \in S$ such that $\lambda_{i-1}<x<\lambda_i$ where, by convention, $\lambda_0<x<\lambda_{d+1}$ for all $x$. One easily verifies that $(S, \lambda) \mapsto (S_1, \ldots, S_{d+1})$ is an equivalence.
\end{proof}

\begin{corollary} \label{cor:OIgrob}
The category $\OI(d)$ is Gr\"obner. In particular, the category of $\OI(d)$-modules is locally noetherian.
\end{corollary}

\begin{proof}
By \cite[Theorem~7.1.2]{catgb} the category $\OI$ is Gr\"obner, and by \cite[Proposition~4.3.5]{catgb} a finite product of Gr\"obner categories is Gr\"obner, so by Proposition \ref{proposition:oidoi} the category $\OI(d)$ is Gr\"obner.  The assertion about finitely generated $\OI(d)$-modules now follows from Theorem \ref{theorem:quasigrobner}.
\end{proof}

\begin{corollary} \label{cor:OIF}
The category $\OI(d)$ satisfies Property~{\rm (F)}. In particular, the tensor product of finitely generated $\OI(d)$-modules is a finitely generated $\OI(d)$-module and the tensor product of $\OI$-modules that are generated in finite degree is also
generated in finite degree.
\end{corollary}

\begin{proof}
The category $\OI$ satisfies Property~(F): this can be proved similarly to \cite[Proposition~7.3.1]{catgb}. One easily sees that a finite product of categories satisfying Property~(F) again satisfies Property~(F), which combined with Proposition \ref{proposition:oidoi} yields the fact that
$\OI(d)$ satisfies Property~(F).  The assertion about tensor products of finitely generated $\OI(d)$-modules now
follows from Proposition \ref{prop:pointwisetensor}, and the assertion about tensor products of $\OI$-modules that
are generated in finite degree follows from Proposition~\ref{prop:finddeg}.
\end{proof}

Finally, we state a result about the growth of finitely generated $\OI$-modules over fields.

\begin{proposition} \label{prop:oidim}
Let $M$ be a finitely generated $\OI$-module over a field $\bk$. Then the function $n \mapsto \dim_\bk M_n$ is a polynomial function for $n \gg 0$.
\end{proposition}

\begin{proof}
By \cite[Theorem 7.1.2]{catgb}, $\OI$ is an ``O-lingual category'', and by \cite[Theorem 6.3.2]{catgb}, this implies the polynomiality statement.
\end{proof}

\subsection{Kan extension}

We now study left Kan extensions along the functor $\Phi \colon \OI(d) \to \OI$.

\begin{proposition} \label{prop:OIkan}
Let $M$ be an $\OI(d)$-module. Then $\Phi_!(M)_n = \bigoplus_{\lambda} M_{n,\lambda}$, where the sum is taken over all increasing $d$-tuples $\lambda$ in $[n]$.
\end{proposition}

\begin{proof}
By \S \ref{ss:kan}, we see that $\Phi_!(M)_n$ is $\colim(M \vert_{\OI(d)_{/[n]}})$. The category $\OI(d)_{/[n]}$ can be viewed as consisting of triples $(S,\mu,f)$, where $(S,\mu) \in \OI(d)$ and $f\colon S \rightarrow [n]$ is a morphism in $\OI$.  For an increasing $d$-tuple $\lambda$ in $[n]$, let $\OI(d)_{/[n],\lambda}$ be the full subcategory of $\OI(d)_{/[n]}$ spanned by triples $(S,\mu,f)$ such that $f$ takes $\mu$ to $\lambda$.  Then $\OI(d)_{/[n]}$ is the disjoint union of its subcategories $\OI(d)_{/[n], \lambda}$. Furthermore, $([n],\lambda,\text{id})$ is the final object of $\OI(d)_{/[n],\lambda}$. The result now follows.
\end{proof}

\begin{corollary}
The functor $\Phi_!$ is exact.
\end{corollary}

\subsection{Shift functors}

Fix a functorial coproduct $\amalg$ on the category of finite sets.  For finite sets $S$ and $T$, we view $S \amalg T$ as the disjoint union of $S$ and $T$; of course, this
requires care when $S$ and $T$ share elements.
Consider the functor $\Sigma_0 \colon \OI(d) \to \OI(d)$ given by $\Sigma_0(S, \lambda)=(S \amalg \{\infty\}, \lambda)$, where $S \amalg \{\infty\}$ is given a total order by setting $x<\infty$ for all $x \in S$. Given an $\OI(d)$-module $M$, we define the {\bf shift} of $M$, denoted $\Sigma(M)$, to be $\Sigma_0^*(M)$. There is a map $(S,\lambda) \to (S \amalg \{\infty\}, \lambda)$ in $\OI(d)$ induced by the inclusion $S \hookrightarrow S \amalg \{\infty\}$.  This
map induces a map $M \to \Sigma(M)$ of $\OI(d)$-modules.  We let $\ol{\Sigma}(M)$ denote the cokernel of this map. We call it the {\bf reduced shift} of $M$. This has the following nice property:

\begin{proposition} \label{prop:redshift}
Suppose that $M$ is an $\OI$-module such that $M_0$ is a finitely generated $\bk$-module and $\ol{\Sigma}(M)$ is a finitely generated $\OI$-module. Then $M$ is a finitely generated $\OI$-module.
\end{proposition}

\begin{proof}
By assumption, we can find $x_1,\ldots,x_m$ with $x_i \in M_{n_i}$ such that
the following holds.
Let $\ox_i \in \Sigma(M)_{n_i-1} \cong M_{n_i}$ be the associated element.
Then the images of $\{\ox_1,\ldots,\ox_m\}$ in $\ol{\Sigma}(M)$ generate
$\ol{\Sigma}(M)$.
We claim that $\{x_1, \dots, x_m\}$ together with a spanning set of $M_0$ 
is a generating set for $M$.  Consider $y \in M_n$ for some $n \geq 0$.  We must
show that $y$ is in the span of the indicated elements.  We will do this by
induction on $n$.  The base case $n=0$ being trivial, we can assume that $n \geq 1$.  
Let $\oy \in \Sigma(M)_{n-1} \cong M_n$ be the associated element.  The image
of $\oy$ in $\ol{\Sigma}(M)_n$ is in the span of the images of
$\{\ox_1,\ldots,\ox_m\}$.  It follows that we can write $y = y' + y''$, where
$y'$ is in the span of $\{x_1,\ldots,x_m\}$ and $y''$ is in the image of
the composition $M_{n-1} \rightarrow \Sigma(M)_{n-1} \cong M_n$.  By induction,
$y''$ is in the span of $\{x_1, \dots, x_m\}$ together with a spanning set of $M_0$,
so $y$ is as well.
\end{proof}

There is a similar functor $\Delta_0 \colon \OI(d-1) \to \OI(d)$ defined by $\Delta_0(S, \lambda)=(S \amalg \{\infty\}, \lambda')$, where $\lambda'$ is obtained by appending $\infty$ to the
end of $\lambda$. For an $\OI(d)$-module $M$, we let $\Delta(M)=\Delta_0^*(M)$, which is an $\OI(d-1)$-module. For $d=0$, we put $\Delta(M)=0$ by convention.

The following result shows how the shift functor interacts with the Kan extension along the functor $\Phi \colon \OI(d) \to \OI$.

\begin{proposition} \label{prop:kanshift}
Let $M$ be an $\OI(d)$-module. Then there is a natural isomorphism
\begin{displaymath}
\Sigma(\Phi_!(M))\cong \Phi_!(\Sigma(M)) \oplus \Phi_!(\Delta(M)).
\end{displaymath}
Moreover, if $\alpha \colon \Phi_!(M) \to \Sigma(\Phi_!(M))$ and $\beta \colon M \to \Sigma(M)$ denote the natural maps, then the diagram
\begin{displaymath}
\xymatrix{
& \Phi_!(M) \ar[ld]_{\alpha} \ar[rd]^{\Phi_!(\beta) \oplus 0} \\
\Sigma(\Phi_!(M)) \ar[rr]^-{\cong} && \Phi_!(\Sigma(M)) \oplus \Phi_!(\Delta(M)) }
\end{displaymath}
commutes. In particular, we have a natural isomorphism
\begin{displaymath}
\ol{\Sigma}(\Phi_!(M)) = \Phi_!(\ol{\Sigma}(M)) \oplus \Phi_!(\Delta(M)).
\end{displaymath}
\end{proposition}

\begin{proof}
Using Proposition \ref{prop:OIkan}, we have
\begin{displaymath}
\Sigma(\Phi_!(M))_n \cong \bigoplus_{\lambda} M_{n+1,\lambda},
\end{displaymath}
where the sum is over all increasing $d$-tuples $\lambda$ in $[n+1]$.  Similarly, we have
\begin{displaymath}
\Phi_!(\Sigma(M))_n \cong \bigoplus_{\lambda} M_{n+1,\lambda},
\end{displaymath}
where the sum is over all increasing $d$-tuples $\lambda$ in $[n]$.  Finally, using
the obvious analogue of Proposition \ref{prop:OIkan} for $\Delta$ we have
\begin{displaymath}
\Phi_!(\Delta(M))_n \cong \bigoplus_{\lambda} M_{n+1,\lambda},
\end{displaymath}
where the sum is over all increasing $d$-tuples $\lambda$ in $[n]$ that end in $n+1$.  
Combining these isomorphisms, we obtain an identification
\begin{displaymath}
\Sigma(\Phi_!(M))_n\cong \Phi_!(\Sigma(M))_n \oplus \Phi_!(\Delta(M))_n.
\end{displaymath}
It is clear that this identification comes from an isomorphism of $\OI$-modules. 
The rest of the proposition follows easily.
\end{proof}

\section{The category \texorpdfstring{$\OVI$}{OVI} and its variants}
\label{s:ovi}

\subsection{Definitions} \label{ss:defovi}
\label{section:ovidefinitions}

Fix a ring $R$ (always assumed to be associative and unital, though not necessarily commutative). Define $\OVI(R)$ to be the following category.  The objects are {\bf ordered free $R$-modules}, that is, pairs $(V, \{v_i\}_{i \in I})$ where $V$ is a finite rank free left $R$-module and $\{v_i\}$ is a basis indexed by a totally ordered set $I$.  The morphisms $(V, \{v_i\}_{i \in I}) \to (W, \{w_j\}_{j \in J})$ are pairs $(f, f_0)$, where $f \colon V \to W$ is a linear map and $f_0 \colon I \to J$ is an order-preserving injection, such that $f(v_i)=w_{f_0(i)}+\sum_{j<f_0(i)} a_{i,j} w_j$ for scalars $a_{i,j}$. In words, $f$ takes the $i$th basis vector of $V$ to the $f_0(i)$th basis vector of $W$ up to ``lower order'' terms. We note that $f_0$ can be recovered from $f$, so it is often omitted. Furthermore, $f$ is necessarily a split injection.  If the ring $R$ is clear, we will just write $\OVI$.

For a non-negative integer $n$, we regard $R^n$ as an ordered free module by endowing it with the standard basis. Every object of $\OVI$ is isomorphic to $R^n$ for a unique $n$. For an $\OVI$-module $M$, we write $M_n$ for its value on $R^n$. The automorphism group of $R^n$ in $\OVI$ is $U_n(R)$, which we denote simply by $U_n$ in this section. It is the subgroup of $\GL_n(R)$ consisting of upper unitriangular matrices.

Let $d$ be a non-negative integer. We define a variant $\OVI(R,d) = \OVI(d)$ as follows. An object is a tuple $(V, \{v_i\}_{i \in I}, \lambda)$ where $(V, \{v_i\}_{i \in I})$ is an ordered free module and $\lambda$ is an increasing $d$-tuple in $I$. A morphism $(V,\{v_i\}_{i \in I},\lambda) \to (W,\{w_j\}_{j \in J},\mu)$ is a morphism $(f,f_0) \colon (V,\{v_i\}) \to (W,\{w_j\})$ in $\OVI$ such that $f_0(\lambda)=\mu$ and such that $f(v_{i})=w_{f_0(i)}$ for all $i$ appearing in $\lambda$ (i.e.,\ no lower terms are allowed on marked basis vectors).

For a tuple $\lambda=(1 \le \lambda_1 < \cdots < \lambda_d \le n)$ we have an object $(R^n, \lambda)$ of $\OVI(d)$. Every object of $\OVI(d)$ is isomorphic to a unique $(R^n, \lambda)$. For an $\OVI(d)$-module $M$, we write $M_{n,\lambda}$ for its value on $(R^n, \lambda)$. We let $U_{n,\lambda}$ be the automorphism group of $(R^n, \lambda)$ in $\OVI(d)$. It is the subgroup of $U_n$ fixing the basis vectors $e_{\lambda_i}$ for $1 \le i \le d$.

\begin{remark}
\label{remark:ovidmotivation}
We introduce $\OVI(d)$ as a technical device for proving Theorem \ref{mainthm2}, which concerns the homology groups $\bH_i(\bU, M)$ for $\OVI(R)$-modules $M$. We will see in Corollary~\ref{cor:PU} that the homology of the principal projective $\OVI$ module at $d$ can be understood in terms of the homology of the trivial $\OVI(d)$-module, a helpful simplification. 
\end{remark}

There are several functors to mention:
\begin{itemize}
\item There is a functor $\OI \to \OVI$ taking a totally ordered set $S$ to the ordered free module $R[S]$ with basis $S$. There is a similar functor $\OI(d) \to \OVI(d)$.

\item There is a functor $\OVI \to \OI$ taking an ordered free module $(V, \{v_i\}_{i \in I})$ to the totally ordered set $I$ and a morphism $(f, f_0)$ to $f_0$. There is a similar functor $\OVI(d) \to \OI(d)$.

\item There is a functor $\Psi \colon \OVI(d) \to \OVI$ given by forgetting $\lambda$. We continue to use the notation $\Psi$ for this functor throughout the paper.
\end{itemize}

We have the following basic fact that follows from interpreting left multiplication by a matrix as a sequence of row operations.

\begin{proposition}
Every morphism $\phi \colon (R^n, \lambda) \to (R^m, \mu)$ in $\OI(d)$ has a unique factorization $\phi = \psi f$ where $\psi \in \Aut(R^m, \mu)$ and $f$ is in the image of the functor $\OI(d) \to \OVI(d)$.
\end{proposition}

\subsection{The case where $R$ is finite}

The purpose of this section is to prove the following fundamental result:

\begin{theorem} \label{thm:qgrob}
If $|R|<\infty$, then the category $\OVI$ is quasi-Gr\"obner. In particular, by Theorem \ref{theorem:quasigrobner} the category $\Rep_{\bk}(\OVI)$ is locally noetherian when $\bk$ is noetherian.
\end{theorem}

\begin{proof}
An {\bf ordered surjection} $f \colon S \to T$ of totally ordered finite sets is a surjection such that for all $i<j$ in $T$ we have $\min f^{-1}(i)<\min f^{-1}(j)$. We let $\OS$ be the category whose objects are finite totally ordered sets and whose morphisms are ordered surjections. This category is known to be Gr\"obner \cite[Theorem~8.1.1]{catgb}.
Given a totally ordered set $S$, we will regard the dual $R[S]^* = \hom_{R}(R[S],R)$ as an element of $\OVI$ as follows.
Let $S^* \subset R[S]^*$ be the dual basis to the basis $S$, and for $s \in S$, write $s^* \in S^*$ for the dual element.  
Then we order $S^*$ via the rule
\begin{equation*}
\tag{\thetheorem.a}
\label{eqn:srule}
s_1^* < s_2^* \quad \quad \text{when} \quad \quad s_2 < s_1.
\end{equation*}
Using this convention, there is a functor $\OS^{\op} \to \OVI$ taking a totally ordered set $S$ to $R[S]^*$ and an ordered surjection $T \rightarrow S$ to the dual of the induced surjective linear map $R[T] \to R[S]$. We will show that this functor satisfies property~(F), which will complete the proof.

Let $V$ be an object of $\OVI$.  Let $T_1,\ldots,T_n \in \OS$ be objects and
$f_i\colon V \rightarrow R[T_i]^*$ be $\OVI$-morphisms such that that the $f_i$ are an enumeration
of all possible morphisms satisfying the following condition:
\begin{compactitem}
\item The set $T_i$ is a total ordering of a finite subset of $V^*$ that spans $V^*$ and $f_i\colon V \rightarrow R[T_i]^*$ is
an $\OVI$-morphism that is dual to the natural surjection $R[T_i] \rightarrow V$.
\end{compactitem}
Since $V$ is finite, there are only finitely many such $f_i$.  Now consider some $S \in \OS$ and an $\OVI$-morphism
$f\colon V \rightarrow R[S]^*$.  To prove that our functor satisfies property~(F), it is enough to prove that
for some $1 \leq i \leq n$ we can write $f = g \circ f_i$, where $g\colon R[T_i]^* \rightarrow R[S]^*$ is dual
to an $\OS$-morphism $S \rightarrow T_i$.  Let $T \subset V^*$ be the image of $S$ under the dual surjection
$f^*\colon R[S] \rightarrow V^*$.  Let $h \colon S \rightarrow T$ be the resulting surjection.  Order $T$ via
the rule
\begin{equation*}
\tag{\thetheorem.b}
\label{eqn:trule}
t_1 < t_2 \quad \quad \text{when} \quad \quad \min h^{-1}(t_1) < \min h^{-1}(t_2),
\end{equation*}
which makes $h$ an $\OS$-morphism.  Combining \eqref{eqn:trule} with \eqref{eqn:srule} (applied to order both $S^*$ 
and $T^*$), we see that $T^*$ has the ordering
\begin{equation*}
\tag{\thetheorem.c}
\label{eqn:tstarrule}
t_1^* < t_2^* \quad \quad \text{when} \quad \quad \max \{s^* \mid s \in h^{-1}(t_1)\} < \max \{s^* \mid s \in h^{-1}(t_2)\};
\end{equation*}
Let $g\colon R[T]^* \rightarrow R[S]^*$ be the $\OVI$-morphism dual to $h$, so
\begin{equation*}
\tag{\thetheorem.d}
\label{eqn:hrule}
g(t^*) = \sum_{s \in h^{-1}(t)} s^* \quad \quad (t \in T).
\end{equation*}
Finally, let $F\colon V \rightarrow R[T]^*$ be the injection dual to the surjection $R[T] \rightarrow V^*$ 
induced by the inclusion $T \hookrightarrow V^*$, so $f = g \circ F$.  The fact that $f$ is an $\OVI$-morphism
together with \eqref{eqn:tstarrule} and \eqref{eqn:hrule} implies that $F$ is an $\OVI$-morphism.
This implies that for some $1 \leq i \leq n$ we have $T = T_i$ and $F = f_i$, and we are done.
\end{proof}

\begin{remark}
By making use of a variant $\OS(d)$ of $\OS$, one can prove a version of the above theorem for $\OVI(d)$. Since we do not need this, we omit the details.
\end{remark}

\subsection{Kan extension}

We now study left Kan extensions along the functor $\Psi \colon \OVI(d) \to \OVI$.

\begin{proposition} \label{prop:OVIkan}
Let $M$ be an $\OVI(d)$-module. Then
\begin{displaymath}
\Psi_!(M)_n = \bigoplus_{\lambda} \Ind_{U_{n,\lambda}}^{U_n}(M_{n,\lambda}),
\end{displaymath}
the sum taken over all increasing sequences $1 \le \lambda_1 < \cdots < \lambda_d \le n$.
\end{proposition}

\begin{proof}
Let $\OVI(d)'$ be the category whose objects are those of $\OVI(d)$ and where a morphism $(V,\{v_i\}_{i \in I},\lambda) \to (W,\{w_j\}_{j \in J},\mu)$ is a morphism $(f,f_0)$ as in $\OVI$ (ignoring the $\lambda$ and $\mu$) such that $f_0$ is a morphism in $\OI(d)$. The automorphism groups in $\OVI(d)'$ are the $U_n$'s. The functor $\Psi$ factors as $\Psi_2 \circ \Psi_1$, where $\Psi_1 \colon \OVI(d) \to \OVI(d)'$ and $\Psi_2 \colon \OVI(d)' \to \OVI$ are the natural functors. Proposition~\ref{prop:kanind} applies to the functor $\Psi_1$, and so we find
\begin{displaymath}
(\Psi_1)_!(M)_{n,\lambda}=\Ind_{U_{n,\lambda}}^{U_n}(M_{n,\lambda}).
\end{displaymath}
Arguing exactly as in the proof of Proposition~\ref{prop:OIkan}, we find
\begin{displaymath}
(\Psi_2)_!(N)_n=\bigoplus_{\lambda} N_{n,\lambda}
\end{displaymath}
for any $\OVI(d)'$-module $N$. The result follows.
\end{proof}

\subsection{$\OVI$-modules and representations of $\bU$}

Define an $\OI(d)$-group $\bU_d$ by $(\bU_d)_{n,\lambda}=U_{n,\lambda}$. If $M$ is an $\OVI(d)$-module then we can regard it as an $\OI(d)$-module via the functor $\OI(d) \to \OVI(d)$, and as such it has the structure of a $\bU_d$-module. We thus have a functor
\begin{displaymath}
\{ \text{$\OVI(d)$-modules} \} \to \{ \text{$\bU_d$-modules} \}.
\end{displaymath}
One can show that the above functor is fully faithful. We do not need this result, so we do not include a proof. We write $\bU$ in place of $\bU_0$.

\section{Noetherianity of \texorpdfstring{$\OVI$}{OVI}-modules} \label{s:noeth}

The goal of this section is to prove Theorem \ref{mainthm3}, which we recall says that if $R$ is a ring
whose underlying additive group is finitely generated and $\bk$ is a commutative noetherian ring, then the category
of $\OVI(R)$-modules over $\bk$ is locally noetherian, that is, any submodule of a finitely generated module is finitely 
generated.  The ring $R$ here is not required to be commutative.  When $R$ is finite, this follows from
the much easier Theorem~\ref{thm:qgrob}.  We will also prove a converse to this result that says that (ignoring degenerate cases) the category $\Rep_{\bk}(\OVI(R))$ is locally noetherian only if $\bk$ is noetherian and the additive 
group of $R$ is finitely generated. We thus have a complete characterization of when $\Rep_{\bk}(\OVI(R))$ is locally noetherian.

This section has four subsections.  We begin in \S \ref{section:thm3outline} by describing a toy version of our proof.
We then prove a technical ring-theoretic result in \S \ref{section:surject}.  The proof of Theorem \ref{mainthm3}
is in the long \S \ref{section:mainthm3proof}.  Finally, in \S \ref{section:converse} we prove the aforementioned
converse to Theorem \ref{mainthm3}. 

\subsection{A toy version of Theorem~\ref{mainthm3}}
\label{section:thm3outline}

In the next sections, we prove Theorem~\ref{mainthm3}. The proof is a bit lengthy and heavy on notation, but the idea behind it is not too complicated. In this section we sketch the proof of a simpler result that illustrates the main ideas.

Theorem~\ref{mainthm3} (with $R=\bZ$) implies that the group algebra $\bk[U_n(\bZ)]$ is left-noetherian, provided $\bk$ is noetherian. Let us try to prove this for $n=3$. The group algebra can be identified, as a $\bk$-module, with
\begin{displaymath}
Q=x_2 y_3 \bk[x_1^{\pm 1}, y_1^{\pm 1}, y_2^{\pm 1}],
\end{displaymath}
which we treat as a $\bk$-submodule of the Laurent polynomial ring in the five variables. The monomials in this module correspond to the group elements in $\bk[U_3(\bZ)]$; the exponents of the $x$'s give the second column, while the exponents of the $y$'s gives the third.

We must show that any $U_3(\bZ)$-submodule of $Q$ is finitely generated. Let $M$ be a given submodule. Let $Q_+$ be the $\bk$-submodule of $Q$ where only positive powers of the variables appear. We would like to associate to $M$ a monomial ideal in $Q_+$, and then use the noetherianity of monomial ideals to conclude that $M$ is finitely generated. By ``ideal'' here we really mean $\bk[x_1,y_1,y_2]$-submodule. The obvious attempt at this is to first form $M_+=M \cap Q_+$ and then take its initial module $\ini(M_+)$, the $\bk$-span of the initial terms of its elements under some monomial order. The problem with this is that $\ini(M_+)$ need not be an ideal. For example, suppose that $M_+$ contains the element $f=x_2 y_3(y_2+1)$, with initial term $\ini(f)=x_2 y_2 y_3$. Let's try to find $x_1 \ini(f)$ in $\ini(M_+)$. If we apply the matrix
\begin{displaymath}
\begin{pmatrix}
1 & 1 & 0 \\
0 & 1 & 0 \\
0 & 0 & 1 \end{pmatrix}
\end{displaymath}
to $f$, we get the element $f'=x_1 x_2 y_3 (y_1 y_2+1)$, with initial term $x_1 x_2 y_1 y_2 y_3$. This is equal to $x_1 y_1 \ini(f)$, so we now need to get rid of the $y_1$. We therefore apply the matrix
\begin{displaymath}
\begin{pmatrix}
1 & 0 & -1 \\
0 & 1 & 0 \\
0 & 0 & 1 \end{pmatrix}
\end{displaymath}
to $f'$, to get the element $f''=x_1 x_2 y_1^{-1} y_3 (y_1 y_2+1)$. This has the correct leading term. However, it no longer belongs to $M_+$: the power of $y$ in the non-leading term is negative. Thus $\ini(f'')$ does not give an element of $\ini(M_+)$. There does not seem to be a way to produce $x_1 \ini(f)$ in $\ini(M_+)$.

\begin{remark}
This approach is really attempting to show that the monoid algebra $Q_+=\bk[U_3(\bZ_{\ge 0})]$ is noetherian. In fact, it is not noetherian. For example, the left ideal generated by the matrices $\begin{pmatrix} 1 & n & 0 \\ 0 & 1 & 1 \\ 0 & 0 & 1 \end{pmatrix}$ for $n \ge 0$ in $Q_+$ is not finitely generated.
\end{remark}

To overcome this problem, we take a more subtle approach. Let $Q_*$ be the submodule of $Q$ where the exponent of $y_2$ is positive, but we still allow negative powers of $x_1$ and $y_1$. Given $M \subset Q$, let $M_*=M \cap Q_*$. We can then form the initial module with respect to $y_2$ (that is, we treat the other variables as constants); call this $\ini_2(M_*)$. Since we allow negative powers of $y_1$, the issue in the previous paragraph does not arise, and $\ini_2(M_*)$ is closed under multiplication by $x_1^{\pm 1}$, $y_1^{\pm 1}$, and $y_2$. We now intersect $\ini_2(M_*)$ with $M_+$ and then take initial terms with respect to $x_1$ and $y_1$. The result is a monomial ideal of $Q_+$. Call this monomial ideal $I(M)$. One can show that if $M \subset M'$ and $I(M)=I(M')$ then $M=M'$. Since $Q_+$ is noetherian as a $\bk[x_1, y_1, y_2]$-module, this proves that $Q$ is noetherian as a $\bk[U_3(\bZ)]$-module.

The same approach works for $\bk[U_n(\bZ)]$, but the process is more involved. Let $Q$ be the group algebra, which we identify with a $\bk$-submodule of the Laurent polynomial ring in variables $x_{i,j}$ with $i \le j$. We let $Q^{(k)}$ be the $\bk$-submodule where the exponents of $x_{i,j}$ with $i \ge k$ are positive. Thus $Q^{(n)}=Q$ and $Q^{(0)}$ is what we would call $Q_+$. Let $M$ be a $U_n(\bZ)$-submodule of $Q$. We obtain a monomial ideal in $Q_+$ as follows: intersect with $Q^{(n-1)}$ and take the initial submodule with respect to $x_{\bullet,n}$; then intersect with $Q^{(n-2)}$ and take the initial submodule with respect to $x_{\bullet,n-1}$; and so on. After $n$ steps we obtain a monomial ideal in $Q_+$. The argument then proceeds as in the previous case.

\begin{remark}
The strategy employed here has some parallels with Hall's proof \cite[Lemma 3]{hall} that the group ring $\bk[\Gamma]$ of a polycyclic group $\Gamma$ is noetherian. There the key point is to take a normal subgroup $\Gamma'$ such that $\Gamma / \Gamma' \cong \bZ$ and treat each element of $\bk[\Gamma]$ as a Laurent polynomial in $x$ with coefficients in $\bk[\Gamma']$ (where $x$ is some generator for $\bZ$) and argue by passing to initial terms.
\end{remark}

The proof for $\OVI(R)$ differs from the above in only two respects. First, there is a great deal of additional bookkeeping. Second, we need a noetherianity result for the kind of $\OI$-monomial ideals that appear in the reduction. This follows easily from Higman's lemma, and is closely related to the theorem (of Cohen \cite{cohen} and Aschenbrenner--Hillar--Sullivant \cite{AH, hillar-sullivant}) that $\bk[x_i]_{i \in \bN}$ is $\Inc(\bN)$-noetherian, where $\Inc(\bN)$ is the monoid of increasing functions $\bN \to \bN$.

\subsection{Eliminating additive torsion}
\label{section:surject}

For technical reasons, Theorem \ref{mainthm3} is easier to prove when $R$ is a ring whose additive group
is a finitely generated free abelian group.  In this section, we show how to reduce to that case.  Our
main tool is the following lemma.

\begin{lemma}
\label{lemma:surjectringnoetherian}
Let $S$ be a ring and let $\bk$ be a commutative ring such that the category of
$\OVI(S)$-modules over $\bk$ is locally noetherian.  Assume that $S$ surjects onto a
ring $R$.  Then the category of $\OVI(R)$-modules over $\bk$ is locally noetherian.
\end{lemma}
\begin{proof}
The surjection $S \rightarrow R$ induces a functor $\Phi\colon \OVI(S) \rightarrow \OVI(R)$.  By
Proposition~\ref{prop:noeth}, it is enough to show that $\Phi$ satisfies property~(F).  For
some $d \geq 1$, let $P_d$ be the principal projective $\OVI(R)$-module associated to $R^d$, so
\[(P_d)_n = \bk[\Hom_{\OVI(R)}(R^d,R^n)] \quad \quad (n \geq 1).\]
By Proposition~\ref{prop:fpullback}, to prove that $\Phi$ satisfies property~(F) it is enough
to prove that $\Phi^{\ast}(P_d)$ is finitely generated.  Since the map $S \rightarrow R$ of rings
is surjective, the induced map
\[\Hom_{\OVI(S)}(S^d,S^n) \rightarrow \Hom_{\OVI(R)}(R^d,R^n)\]
is also surjective for all $n \geq 1$.  This implies that there is a surjective map from the principal projective
$\OVI(S)$-module associated to $S^d$ to $\Phi^{\ast}(P_d)$, and thus that $\Phi^{\ast}(P_d)$ is finitely generated,
as desired.
\end{proof}

\begin{lemma}
\label{lemma:surjectfree}
Let $R$ be a ring whose additive group is finitely generated.  Then there exists a ring $S$
and a surjection $S \rightarrow R$ such that the additive group of $S$ is free and finitely generated.
\end{lemma}
\begin{proof}
Let $R_{\tors}$ be the torsion subgroup of the additive group of $R$ and let $N \geq 1$ be the exponent of $R_{\tors}$, i.e.,\ 
the minimal number such that $N R_{\tors} = 0$.  The proof is by induction on $N$.  In the base case where $N=1$,
the group $R_{\tors}$ is trivial and there is nothing to prove.  Assume, therefore, that $N>1$ and that the
lemma is true for all smaller exponents.  Let $p$ be a prime dividing $N$.  The ring $R/pR$ is a finite ring.
Let $\bZ[R/pR]$ be the monoid ring of the multiplicative monoid underlying $R/pR$, so $\bZ[R/pR]$ consists of
finite sums of formal symbols $\Set{$[x]$}{$x \in R/pR$}$ with the ring structure defined by $[x] [y] = [xy]$.  The
additive group of the ring $\bZ[R/pR]$ is free abelian with basis in bijection with the elements of $R/pR$, and there
exists a ring surjection $\bZ[R/pR] \rightarrow R/pR$ taking $[x] \in \bZ[R/pR]$ to $x \in R/pR$.  Let $R'$ be the fiber product of the surjections $\bZ[R/pR] \rightarrow R/pR$ and $R \rightarrow R/pR$, so we have a cartesian square
\[\xymatrix{
R' \ar[r] \ar[d] & \bZ[R/pR] \ar[d] \\
R \ar[r] & R/pR.}\]
Concretely,
\[R' = \Set{$(x,r) \in \bZ[R/pR] \times R$}{$x$ and $r$ map to same element of $R/pR$}.\]
Since the maps $R \rightarrow R/pR$ and $\bZ[R/pR] \rightarrow R/pR$ are surjective, so is the map $R' \rightarrow R$.
Since the additive group underlying $\bZ[R/pR]$ is torsion-free, the torsion subgroup $(R')_{\tors}$ consists
of pairs $(0,r) \in \bZ[R/pR] \times R_{\tors}$ such that $r \in R_{\tors}$ maps to $0$ in $R/pR$.  It follows that
\[(R')_{\tors} \cong R_{\tors} \cap p R = p R_{\tors}.\]
The exponent of $(R')_{\tors}$ is thus $N/p$, so by induction there exists a ring $S$ whose
additive group is finitely generated and free together with a surjection $S \rightarrow R'$.  The
desired surjection to $R$ is then the composition $S \rightarrow R' \rightarrow R$.
\end{proof}

\subsection{The proof of Theorem \ref{mainthm3}}
\label{section:mainthm3proof}

We now commence with the proof of Theorem \ref{mainthm3},
which we recall says that if $R$ is a ring
whose underlying additive group is finitely generated and $\bk$ is a commutative noetherian ring, then the category
of $\OVI(R)$-modules over $\bk$ is locally noetherian.  By Lemmas \ref{lemma:surjectringnoetherian}
and \ref{lemma:surjectfree}, we can assume that the additive group of $R$ is a finitely generated free abelian
group (this assumption will first be used in Substep \ref{substep:2a} below).  
Fix some $d \geq 0$ and let $P_d$ be the principal projective of $\OVI(R)$ defined by the formula
\[
(P_d)_n = \bk[\Hom_{\OVI(R)}(R^d,R^n)] \quad \quad (n \geq 1).
\]
To prove the theorem, it is enough to prove that the poset of $\OVI(R)$-submodules of $P_d$ is noetherian, i.e.,\ has no 
infinite strictly increasing sequences.  This is trivial for $d=0$, so we can assume that $d \geq 1$.

Say that a map $f\colon I \rightarrow J$ of posets is {\bf conservative} if for all $i,i' \in I$ satisfying
$i \leq i'$ and $f(i) = f(i')$, we have $i = i'$.  If $J$ is a noetherian poset and $f\colon I \rightarrow J$
is a conservative map, then $I$ is also noetherian.  Our strategy will be to use a
sequence of conservative poset maps to reduce proving that the poset of $\OVI(R)$-submodules
of $P_d$ is noetherian to proving that another easier poset $\fM^{(0)}$ is noetherian.
To help the reader understand its structure, we divide our proof into three steps (each
of which is divided into a number of substeps).

Since we will introduce a lot of notation, to help the reader recall the meanings
of symbols we will list the notation that is defined in each substep.

\begin{step}{1}
\label{step:1}
We construct a poset $\fM$ and reduce the theorem to showing that $\fM$ is noetherian.
\end{step}

As in the toy version of our proof, the first step will be to relate the poset
of $\OVI(R)$-submodules of $P_d$ to a poset $\fM$ constructed using certain
``generalized polynomial rings''.  In fact, $\fM$ will be a poset of certain
special $\OI(d)$-submodules of an $\OI(d)$-module $Q$.  
There are three substeps: in Substep \ref{substep:1a} we construct the $\OI(d)$-module
$Q$, in Substep \ref{substep:1b} we construct the poset $\fM$ of special 
$\OI(d)$-submodules of $Q$, and then finally in Substep \ref{substep:1c} we
construct a conservative poset map from the poset of $\OVI(R)$-submodules of $P_d$ to 
$\fM$.

\begin{substep}{1a}
\label{substep:1a}
We construct the $\OI(d)$-module $Q$.

\noindent
{\bf Notation defined:} $\Lambda_n$, $T_{i,j}^r$, $\bT_n$, $\Lambda_n(\bS)$, $\Lambda_{n,\alpha}$, $\bT_{n,\alpha}$, $Q$, $Q_{n,\alpha}$
\end{substep}

We will want to view matrices with entries in $R$ as certain kinds
of ``monomials''.  Since we will be focusing on $P_d$, the relevant matrices will have
$d$ columns and some number $n \geq 1$ of rows.  To that end, we make the following definition:
\begin{compactitem}
\item Define $\Lambda_n$ to be the commutative monoid generated by the set of formal symbols
$T_{i,j}^r$ with $1 \leq i \leq n$ and $1 \leq j \leq d$ and $r \in R$ subject to the
relations $T_{i,j}^{r_1} T_{i,j}^{r_2} = T_{i,j}^{r_1+r_2}$, where $1 \leq i \leq n$ and $1 \leq j \leq d$
and $r_1,r_2 \in R$.
\end{compactitem}
Elements of $\Lambda_n$ are thus ``monomials'' in the $T_{i,j}^r$, and are naturally in bijection with $n \times d$ matrices with entries in $R$: given such a matrix $(r_{i,j})$, the associated element of $\Lambda_n$ is the product of the $T_{i,j}^{r_{i,j}}$, where $i$ ranges over $1 \leq i \leq n$ and $j$ ranges over $1 \leq j \leq d$.  The monoid product in $\Lambda_n$ corresponds to matrix addition. For later use, setting $\bT_n = \Set{$T_{i,j}$}{$1 \leq i \leq n$, $1 \leq j \leq d$}$, for $\bS \subset \bT_n$ we define $\Lambda_n(\bS)$ to be the submonoid of $\Lambda_n$ generated by $\Set{$T_{i,j}^r$}{$T_{i,j} \in \bS$, $r \in R$}$.

Now consider an element $f\in \Hom_{\OVI(R)}(R^d,R^n)$.  By definition, $f$ is a linear map $R^d \rightarrow R^n$ such that there exists a strictly increasing sequence $\alpha = (\alpha_1,\ldots,\alpha_d)$ of $d$ elements of $[n] = \{1,\ldots,n\}$ with the following property:
\begin{compactitem}
\item For $1 \leq i \leq d$, the map $f$ takes the $i$th basis element of $R^d$ to the sum of the $\alpha_i$th basis
element of $R^n$ and an $R$-linear combination of the basis elements of $R^n$ that occur before $\alpha_i$.
\end{compactitem}
Define $\Lambda_{n,\alpha}$ to be the subset of $\Lambda_n$ consisting of elements associated to $n \times d$
matrices of this form.  Defining
\[\bT_{n,\alpha} = \Set{$T_{i,j}$}{$1 \leq j \leq d$, $1 \leq i < \alpha_j$},\]
an element $\tau \in \Lambda_{n,\alpha}$ can be written as 
\begin{equation}
\label{eqn:lambdaalpha}
\tau = T_{\alpha_1,1}^1 T_{\alpha_2,2}^1 \cdots T_{\alpha_d,d}^1 \tau' \quad \quad \text{with $\tau' \in \Lambda_n(\bT_{n,\alpha})$}.
\end{equation}
We thus have a bijection of sets
\[\Hom_{\OVI(R)}(R^d,R^n) \cong \bigsqcup_{\alpha} \Lambda_{n,\alpha},\]
where the disjoint union ranges over the strictly increasing sequences $\alpha$ of $d$ elements of $[n]$.  It follows that
\begin{equation}
\label{eqn:decomposepd}
(P_d)_n = \bk[\Hom_{\OVI(R)}(R^d,R^n)] = \bigoplus_{\alpha} \bk[\Lambda_{n,\alpha}].
\end{equation}
The various $\bk[\Lambda_{n,\alpha}]$ fit together into an $\OI(d)$-module $Q$ with
\[
Q_{n,\alpha} = \bk[\Lambda_{n,\alpha}] \quad \quad ((n,\alpha) \in \OI(d)).
\]

\begin{substep}{1b}
\label{substep:1b}
We construct a poset $\fM$ of $\OI(d)$-submodules of $Q$.

\noindent
{\bf Notation defined:} $\fM$, $E_{i,\alpha_j}^r$
\end{substep}

Consider an $\OVI(R)$-submodule $M$ of $P_d$.  We say that $M$ is a {\bf homogeneous $\OVI(R)$-submodule of $P_d$} if
for all $n \geq 1$, the $\bk$-submodule $M_n$ of $(P_d)_n$ splits according to the decomposition
\eqref{eqn:decomposepd}, i.e.,\ for all $(n,\alpha) \in \OI(d)$ 
there exists some $\bk$-submodule $M_{n,\alpha}$ of $\bk[\Lambda_{n,\alpha}]$ such that
\[M_n = \bigoplus_{\alpha} M_{n,\alpha}.\]
In this case, the various $M_{n,\alpha}$ fit together into an $\OI(d)$-submodule of $Q$.  We thus get
a poset injection
\[\{\text{homogeneous $\OVI(R)$-submodules of $P_d$}\} \hookrightarrow \{\text{$\OI(d)$-submodules of $Q$}\}.\]
The image of this injection consists of all $\OI(d)$-submodules $M$ of $Q$ such that
each $M_{n,\alpha} \subset Q_{n,\alpha}$ is preserved
by the action of $U_n(R)$, which acts on $Q_{n,\alpha}$ via the identification of $Q_{n,\alpha}$ with
the set of formal $\bk$-linear combinations of appropriate $n \times d$ matrices.

For the sake of our later arguments, we will actually consider a larger collection
of submodules.  Define $\fM$ to be the poset of all $\OI(d)$-submodules $M$ of $Q$ such that
the following hold.  Consider $(n,\alpha) \in \OI(d)$ with 
$\alpha = (\alpha_1,\ldots,\alpha_d)$.
Let $\{\vec{e}_1,\ldots,\vec{e}_n\}$ be the standard basis for $R^n$.
For $1 \leq j \leq d$ and $1 \leq i < \alpha_j$ and $r \in R$, define $E_{i,\alpha_j}^r \in U_n(R)$ to be the element
that takes $\vec{e}_{\alpha_j}$ to $r \vec{e}_i + \vec{e}_{\alpha_j}$ and fixes all of the other basis vectors.  We
then require that $M_{n,\alpha}$ be preserved by all of the $E_{i,\alpha_j}^r$ for $1 \leq j \leq d$ and
$1 \leq i < \alpha_j$ and $r \in R$.  The construction in the previous paragraph gives a poset injection
\begin{equation}
\label{eqn:homogeneousmd}
\{\text{homogeneous $\OVI(R)$-submodules of $P_d$}\} \hookrightarrow \fM.
\end{equation}

\begin{substep}{1c}
\label{substep:1c}
We construct a conservative poset map $\{\text{$\OVI(R)$-submodules of $P_d$}\} \longrightarrow \fM$.

\noindent
{\bf Notation defined:} none
\end{substep}

By \eqref{eqn:homogeneousmd}, it is enough to construct a conservative poset map
\begin{equation}
\label{eqn:homogenize}
\{\text{$\OVI(R)$-submodules of $P_d$}\} \rightarrow \{\text{homogeneous $\OVI(R)$-submodules of $P_d$}\}.
\end{equation}

For each $n \geq 1$, put a total ordering on the set of all strictly increasing 
sequences $\alpha$ of $d$
elements of $[n]$ using the lexicographic ordering: $\alpha < \alpha'$ if the first nonzero entry
$\alpha' - \alpha$ is positive.  Given a nonzero element $f \in (P_d)_{n}$, use
the identification \eqref{eqn:decomposepd} to write $f = \sum_{\alpha} f_{n,\alpha}$ with
$f_{n,\alpha} \in \bk[\Lambda_{n,\alpha}]$.  Define $\fin(f) = f_{n,\alpha_0}$, where
$\alpha_0$ is the largest index such that $f_{n,\alpha_0} \neq 0$.

Given an $\OVI(R)$-submodule $M$ of $P_d$ and some $n \geq 1$, define
$\fin(M)_n$ to be the $\bk$-span of $\Set{$\fin(f)$}{$f \in M_n$}$.  It is easy to see
that $\fin(M)$ is also an $\OVI(R)$-submodule of $P_d$.  Moreover, by construction $\fin(M)$
is homogeneous.  The map $M \mapsto \fin(M)$ is thus a poset map as in \eqref{eqn:homogenize}.
We must prove that it is conservative.  Assume otherwise, and let $M$ and $M'$ be $\OVI(R)$-submodules of
$P_d$ such that $M \subsetneq M'$ and $\fin(M) = \fin(M')$.  Let $n \geq 1$ be such that $M_n \subsetneq M'_n$.
Let $x \in M'_n \setminus M_n$ be such that $\fin(x)$ lies in $\bk[\Lambda_{n,\alpha}]$ with $\alpha$ as small
as possible.
Since $\fin(M) = \fin(M')$, we can find some $x' \in M_n$ with $\fin(x) = \fin(x')$.  But then
$x-x' \in M'_n \setminus M_n$, while $\fin(x-x')$ lies in $\bk[\Lambda_{n,\alpha'}]$ with $\alpha'<\alpha$, a
contradiction.

\begin{step}{2}
\label{step:2}
We construct a poset $\fM^{(0)}$ and reduce the theorem to showing that $\fM^{(0)}$ 
is noetherian.
\end{step}

In Step \ref{step:1}, we reduced the theorem to showing that the poset $\fM$ constructed
in Substep \ref{substep:1b} is noetherian.  The goal of this step is to construct
a conservative poset map from $\fM$ to a simpler poset $\fM^{(0)}$.  This will be
done in a sequence of steps.  Recall that $\fM$ is a subposet of the poset of 
$\OI(d)$-submodules of an $\OI(d)$-module $Q$.  In Substep \ref{substep:2a} we will
construct an $\OI(d)$-module filtration
\[Q^{(0)} \subset Q^{(1)} \subset \cdots \subset Q^{(d)} = Q.\]
Next, in Substeps \ref{substep:2b} and \ref{substep:2c}
we will construct two posets $\fM^{(k)}$ and $\fN^{(k)}$ of special $\OI(d)$-submodules of
$Q^{(k)}$ such that $\fM^{(d)} = \fM$.  
Finally, in Substeps \ref{substep:2d} and \ref{substep:2e} we will
construct a sequence of conservative poset maps
\[\fM = \fM^{(d)} \rightarrow \fN^{(d-1)} \rightarrow \fM^{(d-1)} \rightarrow \fN^{(d-2)} \rightarrow \cdots \rightarrow \fN^{(0)} \rightarrow \fM^{(0)}.\]
This reduces the theorem to showing that the poset $\fM^{(0)}$ is noetherian.

\begin{substep}{2a}
\label{substep:2a}
We construct an $\OI(d)$-module filtration
\[Q^{(0)} \subset Q^{(1)} \subset \cdots \subset Q^{(d)} = Q.\]

\noindent
{\bf Notation defined:} $(R,+) = (\bZ^{\lambda},+)$, $R_{\geq 0}$, $\Lambda_{n,\alpha,k+}$, $\Lambda_{n,\alpha,+}$, $\Lambda_{n,+}$, $\Lambda_{n,+}(\bS)$, $Q^{(k)}$, $Q^{(k)}_{n,\alpha}$
\end{substep}

This step is where we use the fact that the additive group of $R$ is a finitely generated free abelian group.  Fix an identification of this additive group with $\bZ^{\lambda}$ for some $\lambda \geq 1$ such that the multiplicative identity $1 \in R$ is identified with an element of $(\bZ_{\geq 0})^{\lambda}$.  Let $R_{\geq 0}$ be
the submonoid of the additive group of $R$ corresponding to $(\bZ_{\geq 0})^{\lambda}$.  The monoid $R_{\geq 0}$ contains $1 \in R$, but
is not necessarily closed under multiplication.

Consider $(n,\alpha) \in \OI(d)$ with $\alpha = (\alpha_1,\ldots,\alpha_d)$.
For $0 \leq k \leq d$, define $\Lambda_{n,\alpha,k+}$ to be the set of all $\tau \in \Lambda_{n,\alpha}$ 
such that if $T_{i,j}^r$ appears in $\tau$ with $i \geq \alpha_k$, then $r \in R_{\geq 0}$.  For $k=0$, we use the convention
$\alpha_0 = 0$, and we will also frequently omit the $k$, so $\Lambda_{n,\alpha,+}$ is the set of all
$\tau \in \Lambda_{n,\alpha}$ such that if $T_{i,j}^r$ appears in $\tau$, then $r \in R_{\geq 0}$.
We will similarly define $\Lambda_{n,+}$ and $\Lambda_{n,+}(\bS)$ for $\bS \subset \bT_n$.
We then define $Q^{(k)}$ to be the $\OI(d)$-submodule of $Q$ where for all 
$(n,\alpha) \in \OI(d)$, we have
\[Q^{(k)}_{n,\alpha} = \bk[\Lambda_{n,\alpha,k+}].\]
We thus have $Q^{(d)} = Q$.  Moreover, 
\[Q^{(0)}_{n,\alpha} = \bk[\Lambda_{n,\alpha,+}].\]

\begin{substep}{2b}
\label{substep:2b}
For $0 \leq k \leq d$, we construct a subposet $\fM^{(k)}$ of the poset of $\OI(d)$-submodules of $Q^{(k)}$ such that $\fM^{(d)} = \fM$.

\noindent
{\bf Notation defined:} $\fM^{(k)}$, ($a.i_k$), ($a.ii_k$), ($b_k$), ($c_k$)
\end{substep}

We begin with some terminology.  A $\bk$-submodule $X$
of $\bk[\Lambda_n]$ is {\bf homogeneous} with respect to $\bS \subset \bT_n$ if
the following holds for all $x \in X$.  Write
\[x = \sum_{q=1}^{m} \tau_q y_q,\]
where for all $1 \leq q \leq m$ we have the following:
\begin{compactitem}
\item $\tau_q \in \Lambda_n(\bS)$, and the different $\tau_q$ are all distinct.
\item $y_q \in \bk[\Lambda_n(\bT_n \setminus \bS)]$.
\end{compactitem}
We then require that $\tau_q y_q \in X$ for all $1 \leq q \leq m$.
  
Now consider some $0 \leq k \leq d$.  Define $\fM^{(k)}$ to be the set of all 
$\OI(d)$-submodules $M$ of $Q^{(k)}$ such that
for all $(n,\alpha) \in \OI(d)$ with $\alpha = (\alpha_1,\ldots,\alpha_d)$,
the following conditions ($a.i_k$), ($a.ii_k$), ($b_k$), and ($c_k$) hold.  To
simplify our notation, we will set $\alpha_0 = 0$.
\begin{compactitem}
\item[($a$)] The $\bk$-module $M_{n,\alpha} \subset \bk[\Lambda_{n,\alpha,k+}]$ is closed under
multiplication by the following elements:
\begin{compactitem}
\item[($i_k$)] $T_{i,j}^r$ with $k \leq j \leq d$ and $1 \leq i < \alpha_k$ and $r \in R$.
\item[($ii_k$)] $T_{i,j}^r$ with $k+1 \leq j \leq d$ and $\alpha_k \leq i < \alpha_j$ and $r \in R_{\geq 0}$.
\end{compactitem}
\item[($b_k$)] The $\bk$-module $M_{n,\alpha}$ is closed under the operators $E_{i,\alpha_j}^r$ 
with $1 \leq j \leq k$ and $1 \leq i < \alpha_j$ and $r \in R$.
\item[($c_k$)] The $\bk$-module $M_{n,\alpha} \subset \bk[\Lambda_{n,\alpha,k+}]$ is homogeneous
with respect to 
\[
  \Set{$T_{\alpha_j,j}$}{$1 \leq j \leq d$} \cup \Set{$T_{i,j}$}{$k+1 \leq j \leq d$ and $\max(\alpha_k,1) \leq i < \alpha_j$}.
\]
\end{compactitem}

\noindent
We claim that $\fM^{(d)} = \fM$.  Condition ($b_d$) implies that
$\fM^{(d)} \subset \fM$, so
we must only prove that $\fM \subset \fM^{(d)}$.  Consider $M \in \fM$ and
$(n,\alpha) \in \OI(d)$ with $\alpha = (\alpha_1,\ldots,\alpha_d)$.
We must verify that $M_{n,\alpha}$ satisfies the properties above:
\begin{compactitem}
\item For ($a.i_d$), we must show that $M_{n,\alpha}$ is closed under
multiplication by $T_{i,d}^r$ for $1 \leq i < \alpha_d$ and $r \in R$.  But this
can be achieved using the operator $E_{i,\alpha_d}^r$, and by the definition
of $\fM$ the $\bk$-module $M_{n,\alpha}$ is closed under this operator,
so ($a.i_d$) follows.
\item No pairs $(i,j)$ satisfy
the conditions of ($a.ii_d$), so that condition is trivial.  
\item Condition ($b_d$) is a special case of the condition defining $\fM$, so it follows.
\item The set referred to in condition ($c_d$) consists only of
\[\Set{$T_{\alpha_j,j}$}{$1 \leq j \leq d$},\]
and by definition every element of $\bk[\Lambda_{n,\alpha}]$ is homogeneous with respect to these variables (see \eqref{eqn:lambdaalpha}), so that condition follows.
\end{compactitem} 

\begin{substep}{2c}
\label{substep:2c}
For $0 \leq k < d$, we construct a subposet $\fN^{(k)}$ of the poset of $\OI(d)$-submodules of $Q^{(k)}$.

\noindent
{\bf Notation defined:} $\fN^{(k)}$, ($a'.i'_k$), ($a'.ii'_k$), ($b'_k$), ($c'_k$)
\end{substep}

Our definition of $\fN^{(k)}$ will be a slight modification of our definition
of $\fM^{(k)}$.  Define $\fN^{(k)}$ to be the set of all
$\OI(d)$-submodules $N$ of $Q^{(k)}$ such that
for all $(n,\alpha) \in \OI(d)$ with
$\alpha = (\alpha_1,\ldots,\alpha_d)$,
the following conditions ($a'.i'_k$), ($a'.ii'_k$), ($b'_k$), and ($c'_k$) hold.
To simplify our notation, we will set $\alpha_0 = 0$.
\begin{compactitem}
\item[($a'$)] The $\bk$-module $N_{n,\alpha} \subset \bk[\Lambda_{n,\alpha,k+}]$ is closed under
multiplication by the following elements:
\begin{compactitem}
\item[($i'_k$)] $T_{i,j}^r$ with $k+1 \leq j \leq d$ and $1 \leq i < \alpha_k$ and $r \in R$.
\item[($ii'_k$)] $T_{i,j}^r$ with $k+1 \leq j \leq d$ and $\alpha_k \leq i < \alpha_j$ and $r \in R_{\geq 0}$.
\end{compactitem}
\item[($b'_k$)] The $\bk$-module $N_{n,\alpha}$ is closed under the operators $E_{i,\alpha_j}^r$
with $1 \leq j \leq k$ and $1 \leq i < \alpha_j$ and $r \in R$.
\item[($c'_k$)] The $\bk$-module $N_{n,\alpha} \subset \bk[\Lambda_{n,\alpha,k+}]$ is homogeneous
with respect to
\[\Set{$T_{\alpha_j,j}$}{$1 \leq j \leq d$} \cup \Set{$T_{i,j}$}{$k+2 \leq j \leq d$ and $\alpha_{k+1} \leq i < \alpha_j$}.\]
\end{compactitem}

\begin{substep}{2d}
\label{substep:2d}
For $1 \leq k \leq d$, we construct a conservative poset map $\fM^{(k)} \rightarrow \fN^{(k-1)}$.

\noindent
{\bf Notation defined:} none.
\end{substep}

Consider $M \in \fM^{(k)}$, so $M$ is an $\OI(d)$-submodule of $Q^{(k)}$.  Define
$N = M \cap Q^{(k-1)}$.  We claim that $N \in \fN^{(k-1)}$.  This requires
checking the conditions ($a'.i'_{k-1}$), ($a'.ii'_{k-1}$), ($b'_{k-1}$), and ($c'_{k-1}$).
Consider some $(n,\alpha) \in \OI(d)$ with
$\alpha = (\alpha_1,\ldots,\alpha_d)$.
\begin{compactitem}
\item Condition ($a'.i'_{k-1}$) asserts that $N_{n,\alpha}$ is closed under multiplication
by $T_{i,j}^r$ with $k \leq j \leq d$ and $1 \leq i < \alpha_{k-1}$ and $r \in R$.  
This follows from the fact that both $M_{n,\alpha}$ and $Q^{(k-1)}_{n,\alpha}$ are 
closed under multiplication by these elements.
This is immediate for $Q^{(k-1)}_{n,\alpha}$.  For $M_{n,\alpha}$, it follows from
($a.i_k$), which says that $M_{n,\alpha}$ is closed under multiplication by $T_{i,j}^r$
with $k \leq j \leq d$ and $1 \leq i < \alpha_k$ and $r \in R$.
\item Condition ($a'.ii'_{k-1}$) asserts that $N_{n,\alpha}$ is closed under multiplication
by $T_{i,j}^r$ with $k \leq j \leq d$ and $\alpha_{k-1} \leq i < \alpha_j$ and $r \in R_{\geq 0}$.
This follows from the fact that both $M_{n,\alpha}$ and $Q^{(k-1)}_{n,\alpha}$ are 
closed under multiplication by these elements.  This is immediate for $Q^{(k-1)}_{n,\alpha}$.
For $M_{n,\alpha}$, it follows from a combination of ($a.i_k$), which handles the cases
where $\alpha_{k-1} \leq i < \alpha_k$ and gives the stronger conclusion that
we can use $r \in R$ instead of just $r \in R_{\geq 0}$, and ($a.ii_k$), which handles the cases where 
$\alpha_k \leq i < \alpha_j$.  Here one might worry that ($a.ii_k$) requires 
$k+1 \leq j \leq d$ instead of $k \leq j \leq d$; however, the case $j=k$ is not needed
since no $i$ satisfies $\alpha_k \leq i < \alpha_k$.
\item Condition ($b'_{k-1}$) asserts that $N_{n,\alpha}$ is closed under the operators
$E_{i,\alpha_j}^r$ with $1 \leq j \leq k-1$ and $1 \leq i < \alpha_j$ and $r \in R$.  
This follows from the fact that both $M_{n,\alpha}$ and $Q^{(k-1)}_{n,\alpha}$ are
closed under these operators.  This is immediate for $Q^{(k-1)}_{n,\alpha}$.
For $M_{n,\alpha}$, it follows from ($b_k$), which says that $M_{n,\alpha}$ is closed
under the operators $E_{i,\alpha_j}^r$ with
$1 \leq j \leq k$ and $1 \leq i < \alpha_j$ and $r \in R$.
\item Condition ($c'_{k-1}$) asserts that $N_{n,\alpha}$ is homogeneous with respect to
\[
\Set{$T_{\alpha_j,j}$}{$1 \leq j \leq d$} \cup \Set{$T_{i,j}$}{$k+1 \leq j \leq d$ and $\alpha_k \leq i < \alpha_j$}.
\]
Condition ($c_k$) says that $M_{n,\alpha}$ is homogeneous with respect to this same
set, and this homogeneity is preserved when we intersect $M_{n,\alpha}$ with
$Q^{(k-1)}_{n,\alpha}$.
\end{compactitem}
We thus can define a poset map $\fM^{(k)} \rightarrow \fN^{(k-1)}$ taking $M \in \fM^{(k)}$
to $M \cap Q^{(k-1)}$.  We claim that this poset map is conservative.  In fact, it is
even injective.  Indeed, consider $M,M' \in \fM^{(k)}$.  Let $N = M \cap Q^{(k-1)}$
and $N' = M' \cap Q^{(k-1)}$, and assume that $N = N'$.  We claim that $M = M'$.  By
symmetry, it is enough to prove that $M \subset M'$.  Consider 
$(n,\alpha) \in \OI(d)$ and 
$x \in M_{n,\alpha}$.  We must prove that $x \in M'_{n,\alpha}$.  We have $x \in Q^{(k)}_{n,\alpha}$.  Setting
\begin{align*}
\bS &= \Set{$T_{i,j}$}{$1 \leq j \leq d$, $1 \leq i < \alpha_j$, $\alpha_{k-1} \leq i < \alpha_k$} \\
&=   \Set{$T_{i,j}$}{$k \leq j \leq d$, $\alpha_{k-1} \leq i < \alpha_k$},
\end{align*}
there exists some $\tau \in \Lambda_{n,\alpha}(\bS)$ such that $\tau x \in Q^{(k-1)}_{n,\alpha}$.  By
($a.i_k$), we have $\tau x \in M_{n,\alpha}$, and thus $\tau x \in N_{n,\alpha}$.  Since
$N = N' \subset M'$, we deduce that $\tau x \in M'_{n,\alpha}$.  Define 
$\tau^{-1} \in \Lambda_{n,\alpha}(\bS)$ to be the result of replacing all the $T_{i,j}^r$ terms
in $\tau$ with $T_{i,j}^{-r}$.  Another application of ($a.i_k$) shows that 
$\tau^{-1} \tau x = x \in M'_{n,\alpha}$, as desired.
 
\begin{substep}{2e}
\label{substep:2e}
For $0 \leq k \leq d-1$, we construct a conservative poset map $\fN^{(k)} \rightarrow \fM^{(k)}$.

\noindent
{\bf Notation defined:} none.
\end{substep}

Fix some $(n,\alpha) \in \OI(d)$ with $\alpha=(\alpha_1,\ldots,\alpha_d)$.
The most important difference between $\fM^{(k)}$ and $\fN^{(k)}$ is that by
($c_k$) the $\bk$-modules making up $\fM^{(k)}$ must be homogeneous with respect to
\[\bS_{n,\alpha,k} = \Set{$T_{\alpha_j,j}$}{$1 \leq j \leq d$} \cup \Set{$T_{i,j}$}{$k+1 \leq j \leq d$ and $\alpha_k \leq i < \alpha_j$ and $i \geq 1$},\]
while by ($c'_k$) the $\bk$-modules making up $\fN^{(k)}$ must only be homogeneous with
respect to
\[\bS_{n,\alpha,k+1} = \Set{$T_{\alpha_j,j}$}{$1 \leq j \leq d$} \cup \Set{$T_{i,j}$}{$k+2 \leq j \leq d$ and $\alpha_{k+1} \leq i < \alpha_j$}.\]
The main function of our poset map $\fN^{(k)} \rightarrow \fM^{(k)}$ will be to
achieve the needed increase in homogeneity.

For $x \in Q^{(k)}_{n,\alpha}$, we will define an ``initial term'' $\fin(x) \in Q^{(k)}_{n,\alpha}$ as follows.  Define
\[
\bS_{n,\alpha,k}' = \bS_{n,\alpha,k} \setminus \bS_{n,\alpha,k+1} = \Set{$T_{i,j}$}{$k+1 \leq j \leq d$ and $\max(\alpha_k,1) \leq i < \alpha_{k+1}$}.
\]
Recall that $R$ is identified as an additive group with $\bZ^{\lambda}$ and that 
$R_{\geq 0} = (\bZ_{\geq 0})^{\lambda} \subset R$.
Using the identification $R = \bZ^{\lambda}$, we will frequently speak of the coordinates of elements of $R$.
We define a total order on $\Lambda_{n,+}(\bS_{n,\alpha,k}')$ in two steps:
\begin{compactitem}
\item We first order $\bS_{n,\alpha,k}'$ by letting $T_{i,j} < T_{i',j'}$ if either $i < i'$ or if $i = i'$ and $j < j'$.  
\item We then order $\Lambda_{n,+}(\bS_{n,\alpha,k}')$ as follows.  Consider distinct $\tau,\tau' \in \Lambda_{n,+}(\bS_{n,\alpha,k}')$.  Enumerating
the elements of $\bS_{n,\alpha,k}'$ in increasing order as $T_{i_1,j_1},\ldots,T_{i_p,j_p}$, we can uniquely write
\[\tau = T_{i_1,j_1}^{r_1} \cdots T_{i_p,j_p}^{r_p} \quad \text{and} \quad \tau' = T_{i_1,j_1}^{r_1'} \cdots T_{i_p,j_p}^{r_p'}.\]
for some $r_i, r_j' \in R_{\geq 0}$.
Let $1 \leq q \leq p$ be the minimal number such that $r_q \neq r_q'$.  We then say that $\tau<\tau'$ if 
the first nonzero coordinate of $r_q' - r_q \in R = \bZ^{\lambda}$ is positive.
\end{compactitem}
For nonzero $x \in Q^{(k)}_{n,\alpha}$, we can uniquely write
\[x = \sum_{q=1}^{m} \tau_q y_q,\]
where for all $1 \leq q \leq m$ we have the following:
\begin{compactitem}
\item $\tau_q y_q \ne 0$ for all $q$.
\item $\tau_q \in \Lambda_{n,+}(\bS_{n,\alpha_k}')$, and the $\tau_q$ are enumerated in increasing order $\tau_1 < \tau_2 < \cdots < \tau_m$.
\item $y_q \in \bk[\Lambda_n(\bT_n \setminus \bS_{n,\alpha,k}')]$.
\end{compactitem}
We then define $\fin(x) = \tau_m y_m \in Q^{(k)}_{n,\alpha}$. We also set $\fin(0)=0$. We will call $\tau_m$ the {\bf initial variable} of $x$, though we remark that this terminology will not be used again until the final paragraph of this substep.

We now construct the poset map $\fN^{(k)} \rightarrow \fM^{(k)}$ as follows.  Consider $N \in \fN^{(k)}$.  For $(n,\alpha) \in \OI(d)$, define
$\fin(N)_{n,\alpha} \subset Q^{(k)}_{n,\alpha}$ to be the $\bk$-span of $\Set{$\fin(x)$}{$x \in N_{n,\alpha}$}$.
It is easy to see that $\fin(N)$ is an $\OI(d)$-submodule of $Q^{(k)}$.  We claim that
$\fin(N) \in \fM^{(k)}$.  To see this, we must check the conditions ($a.i_k$), ($a.ii_k$),  ($b_k$), and ($c_k$).
Consider some $(n,\alpha) \in \OI(d)$ with $\alpha = (\alpha_1,\ldots,\alpha_d)$.
\begin{compactitem}
\item We delay ($a.i_k$) until the end, so we start by verifying condition ($a.ii_k$), which 
asserts that $\fin(N)_{n,\alpha}$ is closed under multiplication by
$T_{i,j}^r$ with $k+1 \leq j \leq d$ and $\alpha_k \leq i < \alpha_j$ and $r \in R_{\geq 0}$.  This is immediate
from ($a'.ii'_k$), which asserts that $N$ is closed under multiplication by these same elements.
\item Condition ($b_k$) asserts that $\fin(N)_{n,\alpha}$ is closed under the operators $E_{i,\alpha_j}^r$
with $1 \leq j \leq k$ and $1 \leq i < \alpha_j$ and $r \in R$.  Condition ($b'_k$) says that $N_{n,\alpha}$
is closed under these operators.  To prove that this implies that $\fin(N)_{n,\alpha}$ is also closed under
these operators, it is enough to prove that for $x \in Q^{(k)}_{n,\alpha}$, we have 
\[\fin(E_{i,\alpha_j}^r(x)) = E_{i,\alpha_j}^r(\fin(x)).\]
To help the reader understand the argument below, we recommend reviewing the correspondence
between elements of $\Lambda_n$ and $n \times d$ matrices from Substep \ref{substep:1a}.  For nonzero $x$, write
\[x = \sum_{q=1}^{m} \tau_q y_q,\]
where for all $1 \leq q \leq m$ we have the following:
\begin{compactitem}
\item $\tau_q y_q \ne 0$ for all $q$.
\item $\tau_q \in \Lambda_{n,+}(\bS_{n,\alpha,k}')$, and the $\tau_q$ are enumerated in increasing order $\tau_1 < \tau_2 < \cdots < \tau_m$.
\item $y_q \in \bk[\Lambda_n(\bT_n \setminus \bS_{n,\alpha,k}')]$.
\end{compactitem}
Since $i < \alpha_j \leq \alpha_k$, for all $1 \leq q \leq m$ we have
\[E_{i,\alpha_j}^r(\tau_q) = \tau_q \tau_q' \quad \text{and} \quad E_{i,\alpha_j}^r(y_q) = y_q y_q'\]
for some $\tau_q' \in \Lambda_n(\bT_n \setminus \bS_{n,\alpha,k}')$ and 
$y_q' \in \bk[\Lambda_n(\bT_n \setminus \bS_{n,\alpha,k}')]$.  We thus have
\[E_{i,\alpha_j}^r(x) = \sum_{q=1}^m E_{i,\alpha_j}^r(\tau_q) \cdot E_{i,\alpha_j}^r(y_q) = \sum_{q=1}^m \tau_q (\tau_q' y_q y_q')\]
and
\[\fin(E_{i,\alpha_j}^r(x)) = \tau_m (\tau_m' y_m y_m') = E_{i,\alpha_j}^r(\fin(x)),\]
as desired.
\item Condition ($c_k$) asserts that $\fin(N)_{n,\alpha}$ is homogeneous with respect to
\[
\bS_{n,\alpha,k} = \Set{$T_{\alpha_j,j}$}{$1 \leq j \leq d$} \cup \Set{$T_{i,j}$}{$k+1 \leq j \leq d$ and $\alpha_k \leq i < \alpha_j$ and $i \geq 1$}.
\]
By ($c'_k$), the $\bk$-module $N_{n,\alpha}$ is homogeneous with respect to $\bS_{n,\alpha,k+1}$, and
the very definition of $\fin(N)_{n,\alpha}$ is
designed to improve this to $\bS_{n,\alpha,k}$.
\item We now finally verify ($a.i_k$), which asserts that $\fin(N)_{n,\alpha}$ is closed under multiplication by
$T_{i,j}^r$ with $k \leq j \leq d$ and $1 \leq i < \alpha_k$ and $r \in R$.  Condition ($a'.i'_k$) says that
$N_{n,\alpha}$ is closed under multiplication by $T_{i,j}^r$ with $k+1 \leq j \leq d$ and $1 \leq i \leq \alpha_j$ and
$r \in R$, and this is preserved when we pass to $\fin(N)_{n,\alpha}$.  We thus must only verify that $\fin(N)_{n,\alpha}$
is closed under multiplication by $T_{i,k}^r$ with $1 \leq i < \alpha_k$ and $r \in R$.  Consider some
$x \in \fin(N)_{n,\alpha}$.  We must show that $T_{i,k}^r x \in \fin(N)_{n,\alpha}$.  Using the already verified
condition ($c_k$), we can assume that $x = \tau y$ with 
\[
\tau \in \Lambda_{n,\alpha}(\bS_{n,\alpha,k}) \quad \text{and} \quad y \in \bk[\Lambda_{n,\alpha}(\bT_{n} \setminus \bS_{n,\alpha,k})].
\]
Using the already verified condition ($b_k$), we know that $E_{i,\alpha_k}^r(x) \in \fin(N)_{n,\alpha}$.  We then
calculate that
\[E_{i,\alpha_k}^r(x) = E_{i,\alpha_k}^r(\tau y) = E_{i,\alpha_k}^r(\tau) E_{i,\alpha_k}^r(y) = (\tau T_{i,k}^r \tau') \cdot y,\]
where $\tau'$ is a product of elements of $\Set{$T_{i,j'}^{r'}$}{$k+1\leq j'\leq d$, $r' \in R$}$ that depends on $\tau$ and $r$ and $i$ and $k$. Letting $(\tau')^{-1}$ be the result of replacing each $T_{i,j'}^{r'}$ in $\tau'$ with $T_{i,j'}^{-r'}$, our already
verified cases of ($a.i_k$) imply that $\fin(N)_{n,\alpha}$ is closed under multiplication by $(\tau')^{-1}$.  In particular,
\[(\tau')^{-1} \cdot E_{i,\alpha_k}^r(x) = (\tau')^{-1} \cdot (T_{i,k}^r \tau' \tau) \cdot y = T_{i,k}^r \tau y = T_{i,k}^r x \in \fin(N)_{n,\alpha},\]
as desired.
\end{compactitem}
The map $N \mapsto \fin(N)$ is thus a poset map from $\fN^{(k)}$ to $\fM^{(k)}$.

We claim that this is a conservative poset map.  Indeed, consider $N_1,N_2 \in \fN^{(k)}$ such that
$N_1 \subset N_2$ and $\fin(N_1) = \fin(N_2)$.  We must
prove that $N_1 = N_2$.  Assume otherwise.  Let $(n,\alpha) \in \OI(d)$ be such that
$(N_1)_{n,\alpha} \subsetneq (N_2)_{n,\alpha}$.  Pick
$x \in (N_2)_{n,\alpha}$ such that $x \notin (N_1)_{n,\alpha}$ and such that the initial variable (see the second paragraph
of this substep for the definition of this) of $x$ is as small as possible among elements with these properties (this is possible since with the above ordering
$\Lambda_{n,+}(\bS_{n,\alpha,k}')$ does not have any infinite strictly decreasing chains).  Since $\fin(N_1) = \fin(N_2)$, we can find some $x' \in (N_1)_{n,\alpha}$ such that $\fin(x') = \fin(x)$.  But then $x - x' \in (N_2)_{n,\alpha}$ and $x-x' \notin (N_1)_{n,\alpha}$, while the initial variable of $x-x'$ is strictly smaller than the initial variable of $x$, a contradiction.

\begin{step}{3}
\label{step:3}
We prove that $\fM^{(0)}$ is noetherian.
\end{step}

In Step \ref{step:2}, we reduced the theorem to showing that $\fM^{(0)}$ is 
noetherian.  In this step, we will prove this.  Defining
\[\Lambda_+ = \bigsqcup_{(n,\alpha) \in \OI(d)} \Lambda_{n,\alpha,+},\]
in Substep \ref{substep:3a} we first construct a useful partial ordering
on $\Lambda_+$ and prove that it is a well partial ordering (see below
for the definition of this).  In Substep \ref{substep:3b}, we use this
partial ordering to prove that $\fM^{(0)}$ is noetherian.

\begin{substep}{3a}
\label{substep:3a}
We construct a partial ordering on $\Lambda_+$ and prove that it is a well partial ordering.

\noindent
{\bf Notation defined:} none.
\end{substep}

We define a partial ordering on $\Lambda_+$ as follows.  Consider 
$\tau,\tau' \in \Lambda_+$.  We say that $\tau \preceq \tau'$ if the following condition
is satisfied:
\begin{compactitem}
\item Let $(n,\alpha),(n',\alpha') \in \OI(d)$ be such that
$\tau \in \Lambda_{n,\alpha,+}$ and $\tau' \in \Lambda_{n',\alpha',+}$.  We then
require that there exists an $\OI(d)$-morphism 
$\iota\colon (n,\alpha) \rightarrow (n',\alpha')$ and some 
$\tau'' \in \Lambda_{n',\alpha',+}$ such that $\tau' = \tau'' \cdot \iota_{\ast}(\tau)$.
\end{compactitem}
It is clear that this is a partial ordering.

The main goal of this substep (which we will accomplish at the end after a number
of preliminaries) is to prove that this partial ordering on $\Lambda_+$
is a well partial ordering, whose definition is as follows.
A poset $(\fP,\prec)$ is {\bf well partially ordered} if every infinite sequence
of elements of $\fP$ contains an infinite weakly increasing subsequence.
See \cite{KruskalSurvey} for a survey about
well partial orderings.  If $\fP$ and $\fP'$ are posets, then we will endow $\fP \times \fP'$ with the
ordering where $(p_1,p_1') \preceq (p_2,p_2')$ if and only $p_1 \preceq p_2$ and $p_1' \preceq p_2'$.  If $\fP$
and $\fP'$ are both well partially ordered, then so is $\fP \times \fP'$ (quick proof: given an infinite sequence
in $\fP \times \fP'$, first pass to a subsequence to make the first coordinate weakly increasing, then pass
to a further subsequence to make the second coordinate also weakly increasing).

Recall that we have identified the additive group of $R$ with $\bZ^{\lambda}$ and that $R_{\geq 0} = (\bZ_{\geq 0})^{\lambda}$.
Using these identifications, we will speak of the coordinates of elements of $R$ and $R_{\geq 0}$.
Endow the set $R_{\geq 0} \cup \{\spadesuit\}$ with the following partial ordering:
\begin{compactitem}
\item $\spadesuit$ is not comparable to any element of $R_{\geq 0}$.
\item For $r_1,r_2 \in R_{\geq 0}$, let $r_1 \preceq r_2$ if all the coordinates of $r_2 - r_1$ are nonnegative.
\end{compactitem}
Since the usual ordering on $\bZ_{\geq 0}$ is a well partial ordering, the restriction of our partial ordering
to $R_{\geq 0} = (\bZ_{\geq 0})^{\lambda}$ is also a well partial ordering.  From this, it is easy to see that
our partial ordering on $R_{\geq 0} \cup \{\spadesuit\}$ is also a well partial ordering.  The product ordering
on $(R_{\geq 0} \cup \{\spadesuit\})^d$ is thus also a well partial ordering.

Let $\cW$ denote the set of finite words in the alphabet $(R_{\geq 0} \cup \{\spadesuit\})^d$.  Endow $\cW$ with the partial ordering where $w_1,w_2 \in \cW$ satisfy $w_1 \preceq w_2$ if and only if the following condition is satisfied.  Write $w_1 = \ell_1 \cdots \ell_{n}$ and $w_2 = \ell_1' \cdots \ell_{n'}'$ with each $\ell_i$
and $\ell'_{i'}$ an element of $(R_{\geq 0} \cup \{\spadesuit\})^d$.  We then require that there exists a strictly
increasing function $\iota\colon [n] \hookrightarrow [n']$ such that $\ell_i \preceq \ell'_{\iota(i)}$ for all $1 \leq i \leq n$. This partial ordering on $\cW$ is a well partial ordering by Higman's lemma \cite[Theorem 4.3]{HigmanLemma}.

As promised, we now prove that the partial ordering on $\Lambda_+$
defined above is a well partial ordering.
Let $\Psi\colon \Lambda_+ \rightarrow \cW$ be the following set function.
Consider $\tau \in \Lambda_{n,\alpha,+} \subset \Lambda_+$.
Write $\alpha = (\alpha_1,\ldots,\alpha_d)$, and expand out $\tau$ as
\[\tau = \prod_{\substack{1 \leq j \leq d \\ 1 \leq i \leq \alpha_j}} T_{i,j}^{r_{i,j}} \quad \quad (r_{i,j} \in R_{\geq 0}).\]
For $1 \leq j \leq d$ and $1 \leq i \leq \alpha_j$, define $\orr_{i,j} \in R_{\geq 0} \cup \{\spadesuit\}$ via the formula
\[\orr_{i,j} = \begin{cases}
r_{i,j} & \text{if $1 \leq i < \alpha_j$}, \\
\spadesuit & \text{if $i = \alpha_j$}.
\end{cases}\]
We remark that by definition we have $r_{\alpha_j,j} = 1$ for all $1 \leq j \leq d$.  For $1 \leq i \leq n$, we define
\[\ell_i = (\orr_{i,1}, \orr_{i,2}, \ldots, \orr_{i,d}) \in (R_{\geq 0} \cup \{\spadesuit\})^d.\]
Finally, we define
\[\Psi(\tau) = \ell_1 \ell_2 \cdots \ell_n.\]
It is clear that $\Psi$ is injective.  What is more, it is immediate from the
definitions that for all $\tau,\tau' \in \Lambda_+$ we have
\[
\tau \preceq \tau' \quad \text{if and only if } \Psi(\tau) \preceq \Psi(\tau').
\]
The key point here is that if we interpret elements of $\Lambda_+$ as matrices
with $d$ columns and entries in $R_+$, the effect of an $\OI(d)$-morphism on these
matrices is to insert extra rows of zeros.  Since $\Psi$ is injective
and $\cW$ is well partially ordered, so is $\Lambda_+$, as claimed.

\begin{substep}{3b}
\label{substep:3b}
We prove that the poset $\fM^{(0)}$ is noetherian.

\noindent
{\bf Notation defined:} none.
\end{substep}

Let $(\Lambda_+,<)$ be the partially ordered set constructed in Substep \ref{substep:3a}.
By definition, $\fM^{(0)}$ is the poset of all $\OI(R)$-modules $M \subset Q^{(0)}$ such that for all $(n,\alpha) \in \OI(d)$ with $\alpha = (\alpha_1,\ldots,\alpha_d)$,
the $\bk$-module $M_{n,\alpha} \subset \bk[\Lambda_{n,\alpha,+}]$ satisfies the following two
properties:
\begin{compactitem}
\item[($\dagger$)] It is closed under multiplication by $T_{i,j}^r$ for all $1 \leq j \leq d$ and $1 \leq i < \alpha_j$ and 
$r \in R_{\geq 0}$.
\item[($\dagger\dagger$)] It is homogeneous with respect to all the possible $T_{i,j}$, i.e.,\ with respect to
\[
\Set{$T_{i,j}$}{$1 \leq j \leq d$ and $1 \leq i \leq \alpha_j$}.
\]
\end{compactitem}
Property ($\dagger\dagger$) implies that $M_{n,\alpha}$ is spanned as a 
$\bk$-module by elements of the form
$c \cdot \tau$ with $c \in \bk$ and $\tau \in \Lambda_{n,\alpha,+}$.  Property
($\dagger$) implies the following:
\begin{compactitem}
\item[($\dagger\dagger\dagger$)] 
Let $\tau_1 \in \Lambda_{n_1,\alpha_1,+} \subset \Lambda_+$ and
$\tau_2 \in \Lambda_{n_2,\alpha_2,+} \subset \Lambda_+$ and $c \in \bk$ be such that
$c \cdot \tau_1 \in M_{n_1,\alpha_1}$ and $\tau_1 \leq \tau_2$.  Then 
$c \cdot \tau_2 \in M_{n_2,\alpha_2}$.
\end{compactitem}
Now assume for the sake of contradiction that $\fM^{(0)}$ is not noetherian.  Let
\[M_1 \subsetneq M_2 \subsetneq M_3 \subsetneq \cdots\]
be an infinite strictly ascending chain in it.  By ($\dagger\dagger$), for all
$i \geq 1$ there exists some $(n_i,\alpha_i) \in \OI(d)$ and some 
$\tau_i \in \Lambda_{n_i,\alpha_i,+}$ and some $c_i \in \bk$ such that
\begin{equation}
\label{eqn:ourcontradiction}
c_i \cdot \tau_i \in (M_i)_{n_i,\alpha_i} \setminus (M_{i-1})_{n_i,\alpha_i}.
\end{equation}
Since our partial ordering on $\Lambda_+$ is a well partial ordering, we can replace
our sequence $\{M_i\}_{i=1}^{\infty}$ with a subsequence and assume that
\[\tau_1 \leq \tau_2 \leq \tau_3 \leq \cdots.\]
For $i \leq i'$, condition ($\dagger\dagger\dagger$) implies that
\[c_i \cdot \tau_{i'} \in (M_i)_{n_{i'},\alpha_{i'}}.\]
For all $q \ge 1$, applying this repeatedly with $i'=q+1$ we see that for
all $1 \leq q' \leq q$ we have
\[
c_{q'} \cdot \tau_{q+1} \in (M_{q'})_{n_{q+1},\alpha_{q+1}} \subset (M_q)_{n_{q+1},\alpha_{q+1}}.
\]
Defining $I_q$ to be the ideal of $\bk$ generated by $\{c_1,\ldots,c_q\}$, this implies
that for all $d \in I_q$ we have
\[
d \cdot \tau_{q+1} \in (M_q)_{n_{q+1},\alpha_{q+1}}.
\]
Since $\bk$ is noetherian, we can pick $q \gg 0$ such that $I_q = I_{q+1}$; in particular,
$c_{q+1} \in I_q$.  But this implies that
\[c_{q+1} \cdot \tau_{q+1} \in (M_q)_{n_{q+1},\alpha_{q+1}},\]
contradicting \eqref{eqn:ourcontradiction}.

\subsection{A converse to Theorem \ref{mainthm3}}
\label{section:converse}

We now prove a converse to Theorem~\ref{mainthm3}:

\begin{proposition} \label{prop:noethconverse}
Let $R$ be a ring and $\bk$ be a commutative ring such that the category of
$\OVI(R)$-modules over $\bk$ is locally noetherian.  Then $\bk$ is noetherian
and the additive group of $R$ is finitely generated.
\end{proposition}

\begin{proof}
Let $P$ be the principal projective $\OVI(R)$-module associated to $R^2$ and let
$P^+$ be the submodule of $P$ generated by all elements lying in $P_n$ with $n>2$.
Then $P/P^+$ is a finitely generated $\OVI(R)$-module with
\[(P/P^+)_n = \begin{cases}
\bk[U_2(R)] & \text{if $n=2$},\\
0 & \text{otherwise}.
\end{cases}\]
It follows that an $\OVI(R)$-submodule of $P/P^+$ is exactly the same thing as a left
ideal in $\bk[U_2(R)]$, so $\bk[U_2(R)]$ is a left-Noetherian
ring.  The group $U_2(R)$ is simply the additive
group underlying $R$, so the proposition follows from the following lemma.
\end{proof}

\begin{lemma}
Let $\bk$ be a commutative ring and let $A$ be an
abelian group such that $\bk[A]$ is 
noetherian.  Then $\bk$ is noetherian and $A$ is finitely generated.
\end{lemma}

\begin{proof}
Since $\bk$ is a quotient of the noetherian ring $\bk[A]$ via the augmentation homomorphism, it is noetherian. For a subgroup $B$ of $A$, let $I_B$ be the ideal of $\bk[A]$ generated by $[b]-[0]$ with $b \in B$. Then $\bk[A]/I_B=\bk[A/B]$, and so $B$ can be recovered from $I_B$ as the elements $b \in A$ such that $[b]-[0] \in I_B$. Suppose that $B_{\bullet}$ is an ascending chain of subgroups of $A$. Then $I_{B_{\bullet}}$ is an ascending chain of ideals in $\bk[A]$ and thus stabilizes. Thus the chain $B_{\bullet}$ stabilizes as well, and so $A$ is noetherian (and thus finitely generated) as an abelian group.
\end{proof}

\section{Homology of \texorpdfstring{$\OVI$}{OVI}-modules} \label{s:hovi}

In this section, $R$ denotes a (not necessarily commutative) ring whose additive group is a finitely generated abelian group and $\bk$ denotes a commutative noetherian ring.  Our goal is to prove Theorem~\ref{mainthm2} from the introduction,
which says that if $M$ is a finitely generated $\OVI$-module then $\bH_i(\bU,M)$ is a finitely generated $\OI$-module
for all $i \geq 0$.  This theorem is proved in \S \ref{section:mainthm2proof} below after some preliminaries.  We then prove in \S \ref{section:mainthm2variant} an analogue of Theorem \ref{mainthm2} 
where we allow upper triangular matrices that are not necessarily unipotent.

\subsection{Homology of some $\OI$-groups}
\label{ss:oihom}

Recall that a group $\Gamma$ is of type $\rF\rP$ over $\bk$ if the trivial $\bk[\Gamma]$-module $\bk$ admits a projective resolution $\bP_\bullet$ such that each $\bP_i$ is a finitely generated $\bk[\Gamma]$-module. In fact, it is equivalent to ask that each $\bP_i$ be a finitely generated free module; see \cite[Theorem VIII.4.3]{brown}.  Many natural classes of groups are of type $\rF\rP$ including finite groups, finitely generated abelian groups, and lattices in semisimple Lie groups. See \cite[Chapter VIII]{brown} for more information.

\begin{proposition} \label{prop:oigphom}
Let $A$ be a group of type $\rF\rP$ over $\bk$ and let $\bE$ be the $\OI$-group $[n] \mapsto A^n$.  Let $M$ be an $\bE$-module which is finitely generated as an $\OI$-module. 
The following then hold.
\begin{compactenum}[\indent \rm (a)]
\item The $\OI$-module $\bH_i(\bE, M)$ is finitely generated for all $i \geq 0$.
\item Suppose $A$ is abelian. Let $C \subset A$ be a finite index subgroup, let $A^n_C$ denote the subgroup $\{(a_1,\dots,a_n) \in A^n \mid a_1 + \cdots + a_n \in C\}$, and let $\bE_C$ be the $\OI$-group $[n] \mapsto A^n_C$. Then the $\OI$-module $\bH_i(\bE_C, M)$ is finitely generated for all $i \geq 0$.
\end{compactenum}
\end{proposition}

\begin{proof}
Pick a free resolution $\bF_\bullet$ of the $\bk[A]$-module $\bk$ such that each $\bF_i$
is a finitely generated $\bk[A]$-module and such that $\bF_0 = \bk[A]$.
For each $n \geq 0$, the complex $(\bF^{\otimes n})_{\bullet}$ is a free resolution of the $\bk[A^n]$-module $\bk$.

For each $i \geq 0$, we assemble the $i^{\text{th}}$ terms of $(\bF^{\otimes n})_{\bullet}$ into an $\OI$-module
$X(i)$ as follows.  First, define
\[
X(i)_n = (\bF^{\otimes n})_i = \bigoplus_{i_1 + \cdots + i_n = i} \bF_{i_1} \otimes \cdots \otimes \bF_{i_n}.
\]
Next, given an $\OI$-morphism $f\colon [n] \rightarrow [m]$, define 
$f_{\ast}\colon X(i)_n \rightarrow X(i)_m$ in the following way.  Consider
a summand $\bF_{i_1} \otimes \cdots \otimes \bF_{i_n}$ of $X(i)_n$.  For
$1 \leq a' \leq m$, define
\[i_{a'}' = \begin{cases}
i_{a} & \text{if $a' = f(a)$ for some $a \in [n]$},\\
0 & \text{otherwise}.\end{cases}\]
We thus obtain a summand $\bF_{i_1'} \otimes \cdots \otimes \bF_{i_m'}$ of
$X(i)_m$.  Define $f_{\ast}\colon X(i)_n \rightarrow X(i)_m$ to be the map that
takes $\bF_{i_1} \otimes \cdots \otimes \bF_{i_n}$ to 
$\bF_{i_1'} \otimes \cdots \otimes \bF_{i_m'}$ by inserting terms that equal
$1 \in \bk[A] = \bF_0$ into the needed places.

For each $i \geq 0$, define $Y(i)$ to be the $\OI$-module $[n] \mapsto (X(i)_n \otimes M_n)_{A^n}$, where
the subscript indicates that we are taking the $A^n$-coinvariants.
The $Y(i)$ form a complex
\[\cdots \longrightarrow Y(3) \longrightarrow Y(2) \longrightarrow Y(1) \longrightarrow Y(0) \longrightarrow 0\]
of $\OI$-modules, and the $\OI$-module $\bH_i(\bE, M)$ is the $i^{\text{th}}$ homology group of this complex.
By the local noetherianity of $\OI$ (Corollary \ref{cor:OIgrob}), to prove that
$\bH_i(\bE,M)$ is a finitely generated $\OI$-module for all $i \geq 0$, 
it is enough to prove that each $Y(i)$ is a finitely generated $\OI$-module, which we now
do.

For each $i \geq 0$, the $\OI$-module $X(i)$ is generated in finite degree (in fact, only terms of degree at most $i$
are needed).  Since $M$ is finitely generated as an $\OI$-module, it is in particular generated in finite degree,
so by Corollary \ref{cor:OIF} the $\OI$-module $X(i) \otimes M$ is also generated in finite degree.  This
implies that $Y(i)$ is also generated in finite degree.
Since $\bF_i$ is a finitely generated $\bk[A]$-module for each $i \geq 0$ and $M_n$ is a $\bk[A^n]$-module that
is finitely generated as a $\bk$-module for each $n \geq 0$, 
it follows that the $\bk$-module
$Y(i)_n = ((\bF^{\otimes n})_i \otimes M_n)_{A^n}$ is a finitely generated $\bk$-module for
all $i,n \geq 0$.
Combining this with the fact that each $Y(i)$ is generated in finite degree, we deduce that 
the $\OI$-module $Y(i)$ is finitely generated for all $i \geq 0$, as desired.

For the second statement, the restriction of $\bF^{\otimes n}$ to $A^n_C$ is still finitely generated since $A^n_C$ is a finite index subgroup in $A^n$, and we can proceed as before.
\end{proof}

\begin{proposition} \label{prop:oigphom2}
Let $A$ be a group of type $\rF\rP$ over $\bk$ and let $\bE'$ be the $\OI(d)$-group given by $\bE'_{n,\lambda}=A^n$. Then $\bH_i(\bE', \ul{\bk})$ is a finitely generated $\OI(d)$-module for
all $i \geq 0$.
\end{proposition}

\begin{proof}
The $\OI(d)$-group $\bE'$ is the pullback of the $\OI$-group $\bE$ from Proposition~\ref{prop:oigphom} through the forgetful functor $\Phi\colon \OI(d) \to \OI$. Thus $\bH_i(\bE', \ul{\bk})$ is the pullback to $\OI(d)$ of the $\OI$-module $\bH_i(\bE, \ul{\bk})$, which is finitely generated by that proposition. The result now follows from the fact that $\Phi$ satisfies Property~(F), which follows easily from Proposition \ref{proposition:oidoi}.
\end{proof}

\subsection{A filtration} \label{ss:hovi}

Our goal in this section is to prove the following result. Recall that $\ol{\Sigma}$ is the reduced shift functor on $\OI$-modules, i.e.,\ the cokernel of the canonical map $M \to \Sigma(M)$. Also, $P_d$ is the principal projective $\OVI$-module associated to the object $R^d$ of $\OVI$.

\begin{proposition} \label{prop:hovi}
The $\OI$-module $\ol{\Sigma}(\bH_i(\bU, P_d))$ has a filtration where the graded pieces are subquotients of $\OI$-modules of the form $\bH_i(\bU, P_e)$ with $e<d$ or $\bH_j(\bU, M)$ with $j<i$ and $M$ a finitely generated $\OVI$-module.
\end{proposition}

We begin with a number of lemmas. Recall that $\Phi \colon \OI(d) \to \OI$ and $\Psi \colon \OVI(d) \to \OVI$ are the forgetful functors.  Also,
$\bU_d$ is the $\OI(d)$-group $(\bU_d)_{n,\lambda}=U_{n,\lambda}$, where $U_{n,\lambda}$ is the group discussed in \S \ref{section:ovidefinitions}.
Finally, the subscript $!$ is used to denote the left Kan extension discussed in \S \ref{ss:kan}.

\begin{lemma} \label{lem:shapiro}
Let $M$ be an $\OVI(d)$-module. We have an isomorphism of $\OI$-modules $\Phi_!(\bH_i(\bU_d, M)) \cong \bH_i(\bU, \Psi_!(M))$.
\end{lemma}

\begin{proof}
Recall from Proposition~\ref{prop:OVIkan} that
\begin{displaymath}
\Psi_!(M)_n = \bigoplus_{\lambda} \Ind_{U_{n,\lambda}}^{U_n}(M_{n,\lambda}).
\end{displaymath}
Thus, by Shapiro's lemma we have
\begin{displaymath}
\bH_i(\bU, \Psi_!(M))_n = \rH_i(U_n, \Psi_!(M)_n) = \bigoplus_{\lambda} \rH_i(U_{n,\lambda}, M_{n,\lambda}),
\end{displaymath}
and this is exactly $\Phi_!(\bH_i(\bU_d, M))$ by Proposition~\ref{prop:OIkan}.  This shows that $\Phi_!(\bH_i(\bU_d, M))$ and $\bH_i(\bU, \Psi_!(M))$ agree on objects, and a moment's reflection shows that they also agree on morphisms.
\end{proof}

\begin{corollary} \label{cor:PU}
We have $\bH_i(\bU, P_d)=\Phi_!(\bH_i(\bU_d, \ul{\bk}))$.
\end{corollary}

\begin{proof}
Let $x = (R^d, \{e_i\}, \lambda) \in \OVI(d)$ where $e_i$ is the standard basis and $\lambda=(1<2<\cdots<d)$.  Set $y = (R^d, \{e_i\}) \in \OVI$. Then $\Psi_!(P_x)=P_y$ by \eqref{eqn:kan-proj}. Since $x$ is the initial object of $\OVI(d)$, we have $P_x(y)=\bk[\Hom(x,y)]=\bk$ for all $y$, so $P_x=\ul{\bk}$. We thus have $\Psi_!(\ul{\bk})=P_d$.  Using the fact that
$P_y$ is just another name for $P_d$, the result now follows from Lemma~\ref{lem:shapiro} with $M=\ul{\bk}$. 
\end{proof}

Let $\bU'_d=\Sigma(\bU_d)$. This is the $\OI(d)$-group given by $(\bU'_d)_{n,\lambda}=U_{n+1,\lambda}$. The group $U_{n+1,\lambda}$ is the semi-direct product $U_{n,\lambda} \ltimes R^n$, and this description is functorial. More precisely, let $\bE_d$ be the $\OI(d)$-group given by $(\bE_d)_{n,\lambda}=R^n$. We then have homomorphisms of $\OI(d)$-groups $i \colon \bU_d \to \bU_d'$ and $p \colon \bU_d' \to \bU_d$ with $pi=\id$ and $\ker(p)=\bE_d$. We observe that $\bE_d$ is in fact naturally an $\OVI(d)$-group, and thus $\bH_i(\bE_d, \ul{\bk})$ is naturally an $\OVI(d)$-module.   Proposition~\ref{prop:oigphom2} says that $\bH_i(\bE_d, \ul{\bk})$ is finitely generated as an $\OI(d)$-module, so it is also finitely generated as an $\OVI(d)$-module.

\begin{lemma} \label{lem:Ufilt}
The $\OI(d)$-module $\ol{\Sigma}(\bH_r(\bU_d, \ul{\bk}))$ admits a filtration where the graded pieces are subquotients of $\bH_i(\bU_d, \bH_{r-i}(\bE_d, \ul{\bk}))$ with $0 \le i \le r-1$.
\end{lemma}

\begin{proof}
The module $\ol{\Sigma}(\bH_r(\bU_d, \ul{\bk}))$ is the cokernel of the map
\begin{displaymath}
\bH_r(\bU_d, \ul{\bk}) \to \bH_r(\bU'_d, \ul{\bk})
\end{displaymath}
induced by the homomorphism $i \colon \bU_d \to \bU'_d$. The result therefore follows from Proposition~\ref{prop:hochserre}, taking $\bG=\bU_d'$ and $\bK=\bU_d$ and $\bE=\bE_d$.
\end{proof}

Recall that if $M$ is an $\OI(d)$-module, then right before Proposition \ref{prop:kanshift} we defined an $\OI(d-1)$-module $\Delta(M)$.

\begin{lemma} \label{lem:Delta}
We have $\Phi_!(\Delta(\bH_i(\bU_d, \ul{\bk})))=\Phi_!(\bH_i(\bU_{d-1}, \ul{\bk}))$.
\end{lemma}

\begin{proof}
By definition,
\begin{displaymath}
\Delta(\bH_i(\bU_d, \ul{\bk}))_{n,\lambda}
=\bH_i(\bU_d, \ul{\bk})_{[n] \amalg \{\infty\}, \lambda \amalg \{\infty\}}
=\rH_i(U_{[n] \amalg \{\infty\}, \lambda \amalg \{\infty\}}, \ul{\bk}).
\end{displaymath}
Since $\{\infty\}$ is the maximal element of $[n] \amalg \{\infty\}$, we have
\begin{displaymath}
U_{[n] \amalg \{\infty\}, \lambda \amalg \{\infty\}} \cong U_{n,\lambda}.
\end{displaymath}
Thus by Proposition \ref{prop:OIkan} we have
\begin{displaymath}
\Phi_!(\Delta(\bH_i(\bU_d, \ul{\bk}))) = \bigoplus_{\lambda} \rH_i(U_{n,\lambda}, \ul{\bk}),
\end{displaymath}
the sum taken over appropriate $d-1$ tuples $\lambda$.  Again using Proposition \ref{prop:OIkan}, this is exactly $\Phi_!(\bH_i(\bU_{d-1}, \ul{\bk}))$.
\end{proof}

\begin{proof}[Proof of Proposition~\ref{prop:hovi}]
We have $\bH_r(\bU, P_d)=\Phi_!(\bH_r(\bU_d, \ul{\bk}))$ by Corollary~\ref{cor:PU}. Thus by Proposition~\ref{prop:kanshift}, we have
\begin{displaymath}
\ol{\Sigma}(\bH_r(\bU, P_d))= \ol{\Sigma}(\Phi_!(\bH_r(\bU_d, \ul{\bk}))) = \Phi_!(\ol{\Sigma}(\bH_r(\bU_d, \ul{\bk}))) \oplus \Phi_!(\Delta(\bH_r(\bU_d, \ul{\bk}))).
\end{displaymath}
By Lemma~\ref{lem:Delta}, the second term on the right is $\Phi_!(\bH_r(\bU_{d-1}, \ul{\bk}))$.  By Corollary~\ref{cor:PU}, this equals $\bH_r(\bU, P_{d-1})$. By Lemma~\ref{lem:Ufilt}, the first term admits a filtration where the graded pieces are subquotients of $\Phi_!(\bH_i(\bU_d, \bH_{r-i}(\bE_d, \ul{\bk})))$ with $0 \le i \le r-1$. Setting $N_i=\bH_{r-i}(\bE_d, \ul{\bk})$, Proposition~\ref{prop:oigphom2} implies that $N_i$ is a finitely generated $\OVI(d)$-module.  Set $M_i=\Psi_!(N_i)$, so $M_i$ is a finitely generated $\OVI$-module. By Lemma~\ref{lem:shapiro}, we have
\begin{displaymath}
\Phi_!(\bH_i(\bU_d, N_i))=\bH_i(\bU, M_i).
\end{displaymath}
Combining all of the above, $\ol{\Sigma}(\bH_r(\bU, P_d))$ admits a filtration where one graded piece is $\bH_r(\bU, P_{d-1})$ and the other graded pieces are subquotients of $\bH_i(\bU, M_i)$ for $0 \le i \le r-1$. The result follows.
\end{proof}

\subsection{Proof of Theorem~\ref{mainthm2}}
\label{section:mainthm2proof}

We now prove Theorem~\ref{mainthm2}. Recall the statement: if $R$ is a ring whose additive group is a finitely generated abelian group, $\bk$ is a commutative noetherian ring, and $M$ is a finitely generated $\OVI$-module,
then $\bH_i(\bU,M)$ is a finitely generated $\OI$-module for all $i \geq 0$. Fix such $\bk$ and $R$ for the rest of this section. Consider the following statement:
\begin{itemize}
\item[$(S_i)$] For a finitely generated $\OVI$-module $M$, the $\OI$-module $\bH_i(\bU,M)$ is finitely generated.
\end{itemize}
Let $i$ be given and suppose that $(S_j)$ is true for all $j<i$ (a vacuous condition
if $i=0$). We will prove $(S_i)$, and this will establish the theorem.

We first show by induction on $d$ that $\bH_i(\bU,P_d)$ is a finitely generated $\OI$-module
for all $d$.
Suppose therefore that $\bH_i(\bU,P_e)$ is a finitely generated $\OI$-module for $e<d$ (a vacuous condition
for $d=0$), and let us prove that $\bH_i(\bU,P_d)$ is a finitely generated $\OI$-module. By Proposition~\ref{prop:hovi}, the $\OI$-module $\ol{\Sigma}(\bH_i(\bU,P_d))$ has a filtration where each graded piece is a subquotient of an $\OI$-module of the form $\bH_i(\bU,P_e)$ with $e<d$ or $\bH_j(\bU,M)$ with $j<i$ and $M$ finitely generated. By the two inductive hypotheses in force, both of these kinds of $\OI$-modules are finitely generated.  Using the local noetherianity of $\OI$-modules (Corollary~\ref{cor:OIgrob}), it follows that $\ol{\Sigma}(\bH_i(\bU,P_d))$ is a finitely generated $\OI$-module.
By Proposition~\ref{prop:redshift}, this implies that the $\OI$-module $\bH_i(\bU,P_d)$ is finitely generated, as desired.  

Let $M$ be a finitely generated $\OVI$-module. Consider an exact sequence
\begin{displaymath}
0 \to K \to P \to M \to 0
\end{displaymath}
where $P$ is a finite direct sum of principal projective $\OVI$-modules.  Since the category of $\OVI$-modules is locally noetherian (Theorem~\ref{mainthm3}), the $\OVI$-module $K$ is finitely generated.  We obtain an exact sequence
\begin{displaymath}
\bH_i(\bU,P) \to \bH_i(\bU,M) \to \bH_{i-1}(\bU,K).
\end{displaymath}
By the previous paragraph, the $\OI$-module $\bH_i(\bU,P)$ is finitely generated. By our inductive hypothesis $(S_{i-1})$, the $\OI$-module $\bH_{i-1}(\bU,K)$ is finitely generated.  Using
the local noetherianity of $\OI$ (Corollary \ref{cor:OIgrob}), it follows 
that the $\OI$-module $\bH_i(\bU,M)$ is finitely generated. We have thus established $(S_i)$,
and the proof is complete.

\begin{remark}
The dimension shifting step in the third paragraph above is the only place in the proof of the theorem where the noetherianity of $\OVI$ is used. We never need noetherianity of $\OVI(d)$.
\end{remark}

\begin{remark} \label{rmk:Liehom}
Suppose the additive group of $R$ is a finite rank free abelian group. We outline an alternative way to get finite generation of the $\OI$-module $[n] \mapsto \rH_i(U_n(R); \bk)$. Let $\fu_n(R)$ be the Lie algebra of strictly upper-triangular $n \times n$ matrices over $R$. By \cite[Theorem 4.3]{grunenfelder}, there is a spectral sequence beginning with the Lie algebra homology of $\fu_n(R)$ which converges to $\rH_i(U_n(R); \bk)$. The Lie algebra homology of $\fu_n(R)$ can be computed from the Koszul complex, whose terms are exterior powers of $\fu_n(R)$, and hence are finitely generated $\OI$-modules (this is similar to the $\OI$-structure on $\bF^{\otimes n}$ in the proof of Proposition~\ref{prop:oigphom}). By noetherianity, $\rH_i(U_n(R); \bk)$ is a finitely generated $\OI$-module.
\end{remark}

\subsection{A variant: relaxing unipotence}
\label{section:mainthm2variant}

For each $n$, we let $B_n(R)$ denote the group of upper-triangular invertible $n \times n$ matrices with entries in $R$. We denote the $\OI$-group $[n] \mapsto B_n(R)$ by $\bB$. Also, if $R$ is commutative and $C \subset R^\times$ is a subgroup, then
let $B_n^C(R) \subset B_n(R)$ be the subgroup whose determinant lies in $C$. We denote the $\OI$-subgroup $[n] \mapsto B_n^C(R)$ by $\bB^C$.

The goal of this section is to prove Theorem \ref{theorem:bvifg} below, which is an analogue of Theorem \ref{mainthm2} for $\bB^C$.  This requires the following lemma.

\begin{lemma} \label{lem:Rtimes-fg}
If $R$ is commutative and the additive group of $R$ is finitely generated, then the group of units $R^\times$ is also finitely generated.
\end{lemma}

\begin{proof}
If $R$ is a domain, then it is either a subring of the ring of integers of a number field, in which case the statement follows from the Dirichlet unit theorem, or it is a finite field, in which case there is nothing to prove.

If $R$ is reduced, then we have an injection $R \to \prod_P R/P$ where the product is over the finitely many associated primes of $R$. Thus we have an injection $R^\times \to \prod_P (R/P)^\times$, and hence $R^\times$ is finitely generated.

Finally, in general we have an exact sequence of groups
\[
0 \to \fN(R) \to R^\times \to (R/\fN(R))^\times \to 0,
\]
where $\fN(R)$ is the nilradical of $R$ equipped with the group structure $x*y = x + y + xy$, and the first map takes $x$ to $1+x$. (We note that the right map is surjective since any lift of a unit in $R/\fN(R)$ to $R$ is automatically a unit.) By the previous cases, the abelian group $(R/\fN(R))^\times$ is finitely generated. The fact that the additive group of $R$ is finitely generated implies that $R$ is noetherian, so
$\fN(R)^n=0$ for some $n$.  For each $k$, the $*$ operation on $\fN(R)$ descends to ordinary addition
on $\fN(R)^k/\fN(R)^{k+1}$.  Since the additive group $\fN(R)^k/\fN(R)^{k+1}$ is a subquotient of the finitely generated
additive group of $R$, the additive group $\fN(R)^k/\fN(R)^{k+1}$ is finitely generated.  Lifting additive generators for
$\fN(R)/\fN(R)^2, \fN(R)^2 / \fN(R)^3, \dots, \fN(R)^{n-1}/\fN(R)^n = \fN(R)^{n-1}$ to $\fN(R)$ gives generators for
$\fN(R)$ with respect to the operation $*$.  We conclude that $R^\times$ is a finitely generated group.
\end{proof}

\begin{theorem} \label{theorem:bvifg}
Suppose that $R$ is commutative and $C \subset R^\times$ is a subgroup. If $M$ is a $\bB$-module which is finitely generated as an $\OI$-module, then $\bH_i(\bB^C, M)$ is a finitely generated $\OI$-module for any $i \ge 0$.
\end{theorem}

\begin{proof}
Let $(R^\times)^n_C$ denote the subgroup of $(R^\times)^n$ consisting of sequences whose product lies in $C$.
We have a short exact sequence of groups
\[
1 \to U_n(R) \to B_n(R) \to (R^\times)^n_C \to 1.
\]
The group $R^\times$ is finitely generated by Lemma~\ref{lem:Rtimes-fg}, and thus so is $(R^\times)^n_C$.
The corollary now follows from the the Hochschild--Serre spectral sequence together with Theorem~\ref{mainthm2} and
Proposition~\ref{prop:oigphom}.
\end{proof}

\section{Application to Iwahori subgroups}
\label{s:iwahori}

The goal of this section is to prove Theorem \ref{maintheorem:iwahori}, whose statement
we now recall.  Let $\cO$ be a number ring, let $\fa \subset \cO$ be a nonzero proper ideal, 
and let $\bk$ be a commutative noetherian ring.  For $i \geq 0$, let $X(i)$ be
the $\OI$-module defined by the rule $[n] \mapsto \rH_i(\GL_{n,0}(\cO,\fa), \bk)$.
We must prove that $X(i)$ is a finitely generated $\OI$-module and that if $\bk$
is a field then $\dim X(i)_n$ equals a polynomial in $n$ for $n \gg 0$.  The 
polynomiality assertion follows from the finite generation assertion together with
Proposition \ref{prop:oidim}, so we must only prove that each $X(i)$ is a finitely
generated $\OI$-module.

Define $R = \cO/\fa$ and let $C \subset R^{\times}$ be the image of
$\cO^{\times}$ under the quotient map $\cO \rightarrow R$.  Let
$\GL_n^C(R)$ be the subgroup of $\GL_n(R)$ consisting of matrices whose
determinant lies in $C$.  Strong
approximation (see, e.g., \cite[Chapter 7]{PlatonovRapinchuk}) implies
that the map $\SL_n(\cO) \rightarrow \SL_n(R)$ is surjective.  This
implies that the map $\GL_n(\cO) \rightarrow \GL_n^C(R)$ is surjective,
which implies that the map $\GL_{n,0}(\cO,\alpha) \rightarrow B_n^C(R)$ is surjective.

We thus have a short exact sequence
\[
1 \longrightarrow \GL_n(\cO,\fa) \longrightarrow \GL_{n,0}(\cO,\fa) \longrightarrow B_n^C(R) \longrightarrow 1.
\]
The associated Hochschild--Serre spectral sequence is of the form
\[
\rH_i(B_n^C(R), \rH_j(\GL_n(\cO,\fa),\bk)) \Longrightarrow \rH_{i+j}(\GL_{n,0}(\cO,\fa), \bk) = X(i+j)_n.
\]
Let $M(j)$ be the $\OVI(R)$-module defined by $M(j)_n = \rH_j(\GL_n(\cO,\fa),\bk)$. Naturality of the above spectral sequence induces a spectral sequence
\begin{equation}
\label{eqn:ss}
\rH_i(B_n^C, M(j)) \Longrightarrow X(i+j)
\end{equation}
of $\OI$-modules.

Letting $\FI$ be the category of finite sets and injections, the rule defining $M(j)$ also endows it with an $\FI$-module structure, which is finitely generated by \cite[Theorem D]{CEFN}. The inclusion $\OI \rightarrow \FI$ satisfies Property (F) (see \cite[Theorem 7.1.4]{catgb}), so by Proposition \ref{prop:fpullback} the induced $\OI$-module structure on $M(j)$ is also finitely generated.  This implies in particular that $M(j)$ is a finitely generated $\OVI(R)$-module. Theorem \ref{theorem:bvifg} now implies that $\rH_i(B^C(R),M(j))$ is a finitely generated $\OI$-module. Since the category of $\OI$-modules is locally noetherian (see Corollary~\ref{cor:OIgrob}), we can now deduce from \eqref{eqn:ss} that each $X(i)$ is a finitely generated $\OI$-module, as desired.

\end{document}